\documentclass[12pt]{article}
\usepackage{amsmath,amssymb,amsthm,amsfonts}
\usepackage{epsfig}

\newtheorem{theorem}{Theorem}

\newtheorem{lemma}[theorem]{Lemma}
\newtheorem{problem}[theorem]{Problem}
\newtheorem{observation}[theorem]{Observation}

\newtheorem{megaclaim}{Claim}
\newtheorem{definition}{Definition}
\newcommand{\claim}[2]{\begin{megaclaim}\label{#1} #2 \end{megaclaim}}
\newcommand{\refclaim}[1]{Claim~\ref{#1}}

\newcommand{\cin}{\mbox{Int}}
\newcommand{\dom}{\mbox{dom}}

\newcommand{\vf}{\varphi}
\newcommand{\ob}[1]{{\rm O}$_{\rm #1}$}

\date{}

\title{$5$-list-coloring planar graphs with distant precolored vertices}

\author{Zden\v{e}k Dvo\v{r}\'ak\thanks{Computer Science Institute of Charles University, Prague, Czech Republic.
E-mail: {\tt rakdver@iuuk.mff.cuni.cz}.
Supported by Institute for Theoretical Computer Science (ITI), project 1M0545 of Ministry of Education of Czech Republic,
and by project GA201/09/0197 (Graph colorings and flows: structure and applications) of Czech Science Foundation.
Further work after these grants finished was supported by the Center of Excellence -- Inst. for Theor. Comp. Sci., Prague, project P202/12/G061 of Czech Science Foundation.}
\and
Bernard Lidick\'y\thanks{
  Iowa State University, Ames IA, USA.
  E-mail: {\tt lidicky@iasate.edu}.}
\and
Bojan Mohar\thanks{Department of Mathematics, Simon Fraser University, Burnaby, B.C. V5A 1S6.
  E-mail: {\tt mohar@sfu.ca}.
  Supported in part by an NSERC Discovery Grant (Canada),
  by the Canada Research Chair program, and by the
  Research Grant P1--0297 of ARRS (Slovenia).}~\thanks{On leave from:
  IMFM \& FMF, Department of Mathematics, University of Ljubljana, Ljubljana,
  Slovenia.}
\and
Luke Postle\thanks{University of Waterloo.  E-mail: {\tt lpostle@uwaterloo.ca}.}}

\begin{document}
\maketitle

\begin{abstract}
We answer positively the question of Albertson asking whether every planar graph can be $5$-list-colored even
if it contains precolored vertices, as long as they are sufficiently far apart from each other.  In order to prove this claim,
we also give bounds on the sizes of graphs critical with respect to $5$-list coloring.  In particular,
if $G$ is a planar graph, $H$ is a connected subgraph of $G$ and $L$ is an assignment of lists of colors to the vertices
of $G$ such that $|L(v)|\ge 5$ for every $v\in V(G)\setminus V(H)$ and $G$ is not $L$-colorable, then
$G$ contains a subgraph with $O(|H|^2)$ vertices that is not $L$-colorable.
\end{abstract}

\section{List colorings of planar graphs}

For a graph $G$, a \emph{list assignment} is a function $L$ that assigns a set of colors to each vertex of $G$.
For $v\in V(G)$, we say that $L(v)$ is the \emph{list} of $v$.  An \emph{$L$-coloring} of $G$ is a function $\vf$
such that $\vf(v)\in L(v)$ for every $v\in V(G)$ and $\vf(u)\neq \vf(v)$ for any pair of adjacent vertices $u,v\in V(G)$.
A graph $G$ is \emph{$k$-choosable} if $G$ is $L$-colorable for every list assignment $L$ such that $|L(v)|\ge k$ for each $v\in V(G)$.

A well-known result by Thomassen~\cite{thomassen1994} states that every planar graph is $5$-choosable.  This implies that
planar graphs are $5$-colorable.  Since planar graphs are known to be $4$-colorable~\cite{AppHak1, AppHakKoc}, a natural
question is whether the result can be strengthened.  Voigt~\cite{voigt1993} gave an example of a non-$4$-choosable planar
graph; hence, the vertices with lists of size smaller than $5$ must be restricted in some way.  For example, Albertson \cite{Alb98} asked the following question.

\begin{problem}\label{quest-albertson}
Does there exist a constant $d$ such that whenever $G$ is a planar graph with list assignment $L$ that
gives lists of size one or five to its vertices and the distance between any pair of vertices with lists
of size one is at least $d$, then $G$ is $L$-colorable?
\end{problem}

For usual colorings, Albertson~\cite{Alb98} proved that if $S$ is a set of vertices in a planar graph $G$ that are precolored with colors $1$--$5$ and are at distance at least 4 from each other, then the precoloring
of $S$ can be extended to a 5-coloring of $G$. 
This solved a problem asked earlier by Thomassen~\cite{Thomassen97}\footnote{The problem was posted in the preprint version of \cite{Thomassen97}; the published version refers to it as being solved already.}.
This result does not generalize to 4-colorings even if we have only two precolored vertices (arbitrarily far apart). Examples are given by triangulations of the plane that have precisely two vertices of odd degree. As proved by Ballantine \cite{Ballantine}
and Fisk \cite{Fisk78}, the two vertices of odd degree must have the same color in every 4-coloring. Thus, precoloring them with a different color, we cannot extend the precoloring to a 4-coloring of the whole graph.

Recently, there has been significant progress towards the solution of Albertson's problem, see \cite{Axenovich20111046} and \cite{5choosfar}.
Let us remark that when the number of precolored vertices is also bounded by some constant, then the answer is positive
by the results of Kawarabayashi and Mohar~\cite{kawmoh} on $5$-list-coloring graphs on surfaces.
In this paper, we prove that the answer is positive in general.

\begin{theorem}\label{thm-main}
If $G$ is a planar graph with list assignment $L$ that
gives lists of size one or five to its vertices and the distance between any pair of vertices with lists
of size one is at least\/ $20780$, then $G$ is $L$-colorable.
\end{theorem}

In the proof, we need the following result concerning the case that the precolored vertices form a connected subgraph,
which is of independent interest.

\begin{theorem}\label{thm-connsg}
Suppose that $G$ is a planar graph, $H$ is a connected subgraph of $G$ and $L$ is an assignment of lists to the vertices
of $G$ such that $|L(v)|\ge 5$ for $v\in V(G)\setminus V(H)$ (and the lists of vertices of $H$ are arbitrary).  If $G$ is not $L$-colorable, then
$G$ contains a subgraph $F$ with at most $8|V(H)|^2$ vertices such that $F$ is not $L$-colorable.
\end{theorem}

Let us remark that the existence of such a subgraph of bounded size follows from \cite{kawmoh}, but our bound on the
size of $F$ is much better and gives a better estimate on the required distance in Problem~\ref{quest-albertson}.
During the publication process of this paper, Postle and Thomas~\cite{Postle2015} improved the bound of Theorem~\ref{thm-connsg}
to linear, and announced a stronger result~\cite{lukethe,PosThoHyperb} implying that this bound holds even for disconnected $H$;
the latter of course gives another proof of Theorem~\ref{thm-main} (with a different forbidden distance between precolored vertices).

In order to prove Theorem~\ref{thm-main}, we instead consider a more general statement allowing
some lists of smaller size.  For a plane graph $G$ and a face $f$ of $G$, let the \emph{boundary} of $f$ denote the subgraph of $G$
consisting of vertices and edges incident with $f$ (let us recall that a \emph{face} is a maximal connected subset of the plane after
removing the points contained in the drawing of $G$).
Let $G$ be a plane graph, let $P$ be a subpath of the boundary $H$ of the outer face of $G$, and let $X$ be a subset of $V(G)\setminus V(P)$.
For a positive integer $M$, a list assignment $L$ for $G$ is
\emph{$M$-valid with respect to $P$ and $X$} if
\begin{itemize}
\item
$|L(v)|=5$ for $v\in V(G)\setminus (V(H)\cup X)$,
\item
$3\le |L(v)|\le 5$ for $v\in V(H)\setminus (V(P)\cup X)$,
\item
$|L(v)|=1$ for $v\in X$,
\item
the subgraph of $G$ induced by $V(P)\cup X$ is $L$-colorable, and
\item
for every $v\in X$, the vertices of $V(G)\setminus \{v\}$ at distance at most $M$ from $v$ do
not belong to $P$ and have lists of size $5$.
\end{itemize}
If $X=\emptyset$ and $L$ is $0$-valid, we simply say that $L$ is \emph{valid with respect to $P$}.

A key ingredient for our proofs is the following well-known result of Thomassen~\cite{thomassen1994} regarding the coloring of planar graphs from lists of restricted sizes.

\begin{theorem}[\cite{thomassen1994}]\label{thm-thom}
If $G$ is a connected plane graph, $xy$ is an edge incident with the outer face of $G$,
and $L$ is a list assignment that is valid with respect to $xy$, then $G$ is $L$-colorable.
\end{theorem}

There are arbitrarily large non-$L$-colorable graphs with this structure if we allow
a path of length two to be precolored (throughout the paper, the \emph{length} of the path
refers to the number of its edges, i.e., a path of length two has three vertices).
Thomassen~\cite{thom-2007} gave their complete description, see Lemma~\ref{lemma-extthom}.
In Theorem~\ref{thm-prepathw}, we deal with the more general case when $P$ has fixed length $k$.
In particular, we show that if $G$ is a minimal non-$L$-colorable graph satisfying the assumptions of Theorem~\ref{thm-prepathw},
then at most $k-2$ of its vertices incident with the outer face have lists of size at least four.
In conjunction with Theorem~\ref{thm-connsg}, this enables us to bound the size of such graphs
subject to the additional assumption that no two vertices with lists of size three are adjacent.

Next, we use the new approach to $5$-choosability of planar graphs developed in \cite{5choosfar}
to show that we can reduce the problem to the case that only one internal vertex is precolored.
Having established this fact, the following lemma gives the affirmative answer to Problem~\ref{quest-albertson}.

\begin{lemma}\label{lemma-albimp}
There exists a constant $M$ with the following property.
Let $G$ be a plane graph and let $H$ be the boundary of its outer face.  Let $P$ be a (possibly null) subpath of $H$ of length at most one.
For every $x\in V(G)\setminus V(P)$ and every list assignment $L$ that is $M$-valid with respect to $P$ and $\{x\}$
such that no two vertices with lists of size three are adjacent, the graph $G$ is $L$-colorable.
\end{lemma}

We first prove Theorem~\ref{thm-connsg}, in Section~\ref{sec-connsg}.  In Section~\ref{sec-size}, we prove
Theorem~\ref{thm-prepathw}.  In Section~\ref{sec-albimp}, we show that Lemma~\ref{lemma-albimp} implies our main result, Theorem~\ref{thm-main}.
Finally, in Section~\ref{sec-albertson}, we prove Lemma~\ref{lemma-albimp}.

Let us mention that we could also allow different kinds of ``irregularities'' other than just precolored vertices, for example,
precolored triangles or crossings, as long as the irregularity satisfies the condition analogous to Lemma~\ref{lemma-albimp}.
To keep the presentation manageable, we do not give proofs in this full generality and focus on the case of precolored single vertices.

\section{Critical graphs}
\label{sec-connsg}

To avoid dealing with irrelevant subgraphs, we define what a list-coloring critical graph means.
Let $G$ be a graph, let $T$ be a (not necessarily induced) subgraph of $G$ and let $L$ be a list assignment to the vertices
of $V(G)$.  For an $L$-coloring $\vf$ of $T$, we say that \emph{$\vf$ extends to an $L$-coloring of $G$} if there exists
an $L$-coloring of $G$ that matches $\vf$ on $V(T)$.  The graph $G$ is \emph{$T$-critical with respect to the list assignment $L$}
if $G\neq T$ and for every proper subgraph $G'\subset G$ such that $T\subseteq G'$,
there exists a coloring of $T$ that extends to an $L$-coloring of $G'$, but does not extend to an $L$-coloring of $G$.
If the list assignment is clear from the context, we shorten this and say that $G$ is \emph{$T$-critical}.
Note that $G$ is list-critical for the usual definition of criticality if and only if it is $\emptyset$-critical.
Let us also observe that every proper subgraph of a $T$-critical graph that includes $T$ is $L$-colorable, and that it may happen that $G$ is also $L$-colorable.

Let $G$ be a $T$-critical graph (with respect to some list assignment).
For $S\subseteq G$, a graph $G'\subseteq G$ is an \emph{$S$-component} of $G$ if $S$ is a proper subgraph of $G'$, $T\cap G'\subseteq S$ and all edges of
$G$ incident with vertices of $V(G')\setminus V(S)$ belong to $G'$.  For example, if $G$ is a plane graph with $T$
contained in the boundary of its outer face and $S$ is a cycle in $G$ that does not bound a face, then the subgraph of $G$ drawn inside the closed
disk bounded by $S$ (which we denote by $\cin_S(G)$) is an $S$-component of $G$.

Another important example of $S$-components comes from chords.
Given a graph $G$ and a cycle $K\subseteq G$, an edge $uv$ is a \emph{chord} of $K$ if $u,v\in V(K)$, but $uv$ is not an edge of $K$.
For an integer $k\ge 2$, a path $v_0v_1\ldots v_k$ is a \emph{$k$-chord\/} of $K$ if $v_0,v_k\in V(K)$ and
$v_1, \ldots, v_{k-1}\not\in V(K)$.  Suppose that $K$ bounds the outer face of a plane $T$-critical graph $G$,
where $T$ is a subpath of $K$.  Let the set $K'$ consist of $V(K)\setminus V(T)$ and of the endvertices of $T$.
Let $S$ be a chord or a $k$-chord of $K$ such that both its endvertices belong to $K'$, and let $c$ be a simple closed
curve in plane consisting of $S$ and a curve in the outer face of $G$ joining the endpoints of $S$,
such that $T$ lies outside the closed disk bounded by $c$.  The subgraph $G'$ of $G$ drawn inside the closed
disk bounded by $c$ is an $S$-component. We say that $G'$ is the subgraph of $G$ \emph{split off} by $S$.

We are also going to need a generalization of the construction from the preceding paragraph.
In the same situation, let $Q\neq T$ be a path in $G$ such that no internal vertex of $Q$ belongs to $K'$
and both endvertices belong to $K'$ (so $Q$ may possibly contain some edges and vertices of $T$, but $T\not\subseteq Q$).
We say that $Q$ is a \emph{span} of $G$.
Let $c$ be a simple closed curve in plane consisting of $Q$ and a curve in the outer face of $G$ joining the endpoints of $Q$,
such that $T$ is not contained in the closed disk bounded by $c$.  Again, the subgraph $G'$ of $G$ drawn inside the closed
disk bounded by $c$ is a $Q$-component, and we say that $G'$ is the subgraph of $G$ \emph{split off} by $Q$.

The $S$-components have the following basic property.

\begin{lemma}
\label{lemma-crs}
Let $G$ be a $T$-critical graph with respect to a list assignment $L$.  Let $G'$ be an $S$-component of $G$, for some $S\subseteq G$.
Then $G'$ is $S$-critical with respect to $L$.
\end{lemma}

\begin{proof}
If $G$ contains an isolated vertex $v$ that does not belong to $T$, then since $G$ is $T$-critical, we have that $L(v)=\emptyset$
and $T=G-v$.  Observe that if $G'$ is an $S$-component of $G$, then $S\subseteq T$ and $G'-v=S$, and clearly $G'$ is $S$-critical.

Therefore, we can assume that every isolated vertex of $G$ belongs to $T$.  Consequently, every isolated vertex of $G'$ belongs to $S$.
Suppose for a contradiction that $G'$ is not $S$-critical.  Then, there exists an edge $e\in E(G')\setminus E(S)$ such that
every $L$-coloring of $S$ that extends to $G'-e$ also extends to $G'$.  Note that $e\not\in E(T)$.
Since $G$ is $T$-critical, there exists a coloring $\psi$ of $T$ that extends to an $L$-coloring $\vf$ of $G-e$, but does not
extend to an $L$-coloring of $G$.  However, by the choice of $e$, the restriction of $\vf$ to $S$ extends to an $L$-coloring
$\vf'$ of $G'$.  Let $\vf''$ be the coloring that matches $\vf'$ on $V(G')$ and $\vf$ on $V(G)\setminus V(G')$.
Observe that $\vf''$ is an $L$-coloring of $G$ extending $\psi$, which is a contradiction.
\end{proof}

Lemma~\ref{lemma-crs} together with the following reformulation of Theorem~\ref{thm-thom}
enables us to apply induction to critical graphs.

\begin{lemma}\label{lemma-crittw}
Let $G$ be a plane graph with its outer face bounded by a cycle $H$ and let $L$ be a list assignment for $G$ such that $|L(v)|\ge 5$ for $v\in V(G)\setminus V(H)$. 
If $G$ is $H$-critical with respect to the list assignment $L$,
then either $H$ has a chord or $G$ contains a vertex with at least three neighbors in $H$.
\end{lemma}
\begin{proof}
Suppose that $H$ is an induced cycle.
Since $G$ is $H$-critical, there exists an $L$-coloring $\vf$ of $H$ that does not
extend to an $L$-coloring of $G$.  Let $L'$ be the list assignment for the graph $G'=G-V(H)$
obtained from $L$ by removing the colors of vertices of $H$ given by $\vf$ from the lists of their neighbors.
Since $\vf$ does not extend to $G$, it follows that $G'$ is not $L'$-colorable, and by Theorem~\ref{thm-thom},
there exists $v\in V(G')$ with $|L'(v)|\le 2$.  This implies that $v$ has at least three neighbors in $H$.
\end{proof}

Clearly, to prove Theorem~\ref{thm-connsg}, it suffices to bound the size of critical graphs.
It is more convenient to bound the \emph{weight\/} of such graphs, which is defined as follows.
\begin{definition}[Weights of faces and vertices]
Let $G$ be a plane graph, let $P$ be a subgraph of the boundary $H$ of the outer face $f_0$ of $G$, and let $L$ be a list assignment.
For a face $f$ of $G$, let $|f|$ denote the length of $f$ (if an edge is incident with the same face $f$ on both sides, it contributes $2$ to $|f|$).
The weight of $f$ is
$$\omega_{G,P,L}(f)=\begin{cases}
|f|-3&\text{if $f\neq f_0$}\\
0&\text{if $f=f_0$.}
\end{cases}$$
For a vertex $v$ of $G$, the weight is
$$\omega_{G,P,L}(v)=\begin{cases}
1&\text{if $v\in V(P)$ and $v$ is a cut-vertex of $G$}\\
0&\text{if $v\in V(P)$ and $v$ is not a cut-vertex of $G$}\\
|L(v)|-3&\text{if $v\in V(H)\setminus V(P)$}\\
0&\text{if $v\in V(G)\setminus V(H)$.}
\end{cases}$$
\end{definition}
In the cases where $G$, $P$ or $L$ are clear from the context, we drop the corresponding indices in $\omega_{G,P,L}$ notation.
\begin{definition}[Weight of a graph]
Let $G$ be a plane graph, let $P$ be a subgraph of the boundary of the outer face of $G$, and let $L$ be a list assignment.
We set
$$
    \omega_{P,L}(G)=\sum_{v\in V(G)} \omega_{G,P,L}(v) + \sum_{f\in F(G)} \omega_{G,P,L}(f),
$$
where the sums go over the vertices and faces of $G$, respectively.
\end{definition}

Let $\cal S$ be a set of proper colorings of $K$.  We say that $v\in V(K)$ is \emph{relaxed in $\cal S$} if there exist
two distinct colorings in $\cal S$ that differ only in the color of $v$.

\begin{lemma}
\label{lemma-preouf}
Let $G$ be a plane graph with its outer face bounded by a cycle $H$ and let $L$ be a list assignment for $G$ such that $|L(v)|\ge 5$ for $v\in V(G)\setminus V(H)$.
If $G$ is $H$-critical with respect to the list assignment $L$ and $G$ is not equal to $H$ with one added chord, then
$$\omega_{H,L}(G)+\frac{|V(G)\setminus V(H)|}{2|H|+2}\le |H|-9/2.$$
\end{lemma}

\begin{proof}
We proceed by induction. Assume that the lemma holds for all graphs having fewer edges than $G$.
For a subgraph $G'$ of $G$ with outer face bounded by a cycle $C$, let
$$
  \theta(G') = \omega_{C,L}(G') + \frac{|V(G')\setminus V(C)|}{2|H|+2}.
$$
Let $C\ne H$ be a cycle in $G$ such that $|C|\le |H|$. By Lemma \ref{lemma-crs}, $\cin_C(G)$ is $C$-critical with respect to $L$ unless $C$ bounds a face.
If $\cin_C(G)$ has at least four faces (including the outer one bounded by $C$), then the induction hypothesis applied to $\cin_C(G)$ implies that
\begin{eqnarray*}
\theta(\cin_C(G))&=&\omega_{C,L}(\cin_C(G)) + \frac{|V(\cin_C(G))\setminus V(C)|}{2|H|+2}\\
&\le&\omega_{C,L}(\cin_C(G)) + \frac{|V(\cin_C(G))\setminus V(C)|}{2|C|+2}\\
&\le& |C|-9/2.
\end{eqnarray*}
Observe that if $\cin_C(G)$ has three faces (i.e., consists of $C$ and its chord), then
$\theta(\cin_C(G)) = \omega_{C,L}(\cin_C(G)) = |C|-4$, and if $C$ bounds a face of $G$, then $\theta(\cin_C(G))=|C|-3$.

We construct a sequence $G_0\supset G_1\supset \ldots \supset G_k$ of subgraphs of $G$ with outer
faces bounded by cycles $H_0$, $H_1$, \ldots, $H_k$ such that for $0\le i\le k$, $G_i$ is $H_i$-critical and
\begin{equation}
   \omega_{H_i,L}(G_i) = \omega_{H,L}(G)-(|H|-|H_i|).
   \label{eq:H_i}
\end{equation}
We set $G_0=G$ and $H_0=H$.  Suppose that $G_i$ has already been
constructed.  If $H_i$ has a chord, or a vertex of $G_i$ has at least four neighbors
in $H_i$, then we set $k=i$ and stop.
Otherwise, by Lemma~\ref{lemma-crittw}, there is a vertex $v\in V(G_i)$ with three neighbors $v_1$, $v_2$
and $v_3$ in $H_i$. 
Let $C_1$, $C_2$ and $C_3$ be the three cycles of $H_i+\{v_1v,v_2v,v_3v\}$ distinct from $H_i$,
where $C_j$ does not contain the edge $vv_j$ ($j=1,2,3$).  If at most one of these cycles bounds a face of $G_i$,
then we set $k=i$ and stop.  Otherwise, assume that say $C_1$ and $C_3$ bound faces of $G_i$.
Let ${\cal S}_i$ be the set of $L$-colorings of $H_i$ that do not extend to an $L$-coloring of $G_i$.
If $v_2$ is relaxed in ${\cal S}_i$, then again set $k=i$ and stop.  Otherwise, let
$G_{i+1} = \cin_{C_2}(G_i)$ and let $H_{i+1}=C_2$ be the cycle bounding its outer face.

Note that in the last case, we have $|H_{i+1}| \le |H_i|$ and that
\begin{equation}
  |H_{i}| - |H_{i+1}| = (|C_1|-3) + (|C_3|-3).
\label{eq:Hi change}
\end{equation}
Furthermore, if $w\in V(H_{i+1})\setminus \{v\}$ is relaxed in ${\cal S}_i$, then it is also relaxed in ${\cal S}_{i+1}$.  This is
obvious if $w\neq \{v_1,v_3\}$.  Suppose that say $w=v_1$ and that $\vf_1,\vf_2\in {\cal S}_i$
differ only in the color of $v_1$.  Since $v$ has a list of size at least $5$, there exists a color
$c\in L(v)\setminus\{\vf_1(v_1),\vf_2(v_1),\vf_1(v_2),\vf_1(v_3)\}$.
Let $\vf'_1$ and $\vf'_2$ be the $L$-colorings of $H_{i+1}$ that match $\vf_1$ and $\vf_2$
on $H_i$ and $\vf'_1(v)=\vf'_2(v)=c$.  Then neither $\vf'_1$ nor $\vf'_2$ extends to an
$L$-coloring of $G_{i+1}$, showing that $v_1$ is relaxed in ${\cal S}_{i+1}$.  Similarly, $v$ is relaxed in ${\cal S}_{i+1}$,
since for an arbitrary $\vf\in {\cal S}_i$ (the set ${\cal S}_i$ is nonempty, since $G_i$ is $H_i$-critical),
there exist at least two possible colors for $v$ in $L(v)\setminus\{\vf(v_1),\vf(v_2),\vf(v_3)\}$,
giving two elements of ${\cal S}_{i+1}$ that differ only in the color of $v$.
We conclude that the number of non-relaxed vertices in ${\cal S}_{i+1}$ is smaller than
the number of non-relaxed vertices in ${\cal S}_i$ for every $i<k$, and
consequently, $k\le |H|$.

Lemma \ref{lemma-crs} implies that every $G_i$ is $H_i$-critical. It is also easy to see
by induction and using (\ref{eq:Hi change}) that (\ref{eq:H_i}) holds for $0\le i\le k$.
In each step in the construction of the sequence $(G_i,H_i)_{i=0}^k$,
the number $|V(G_i)\setminus V(H_i)|$ is decreased by 1. Thus, (\ref{eq:H_i}) implies that
\begin{equation}
   \theta(G) - \theta(G_k) = |H|-|H_k| + \frac{k}{2|H|+2}.
   \label{eq:Hk theta}
\end{equation}

Suppose that there exists a proper subgraph $G'\supset H_k$ of $G_k$ and a coloring $\vf\in {\cal S}_k$ that
does not extend to an $L$-coloring of $G'$.  We may choose $G'$ to be $H_k$-critical.
Note that
\begin{eqnarray*}
   \theta(G)&=&\frac{k}{2|H|+2}+(|H|-|H_k|)+\theta(G_k)\\
   &=&\frac{k}{2|H|+2}+(|H|-|H_k|)+\theta(G')+\sum_{f\in F(G')\setminus\{H_k\}} (\theta(\cin_f(G))-\omega(f)).
\end{eqnarray*}
By induction, $\theta(G')\le |H_k|-4$,
since $G'\ne H_k$. This implies that
all faces of $G'$ are shorter than $|H|$.  Since $G'$ is a proper subgraph of $G_k$, we have
$\theta(\cin_f(G))\le \omega(f)-1$ for at least one face $f$ of $G'$ by induction.  It follows that
$$\theta(G)\le 1/2+(|H|-|H_k|)+(|H_k|-4)-1=|H|-9/2,$$ as required.  Therefore, we can assume that every coloring in ${\cal S}_k$
extends to every proper subgraph of $G_k$ that includes $H_k$.

Let us now consider various possibilities in the definition of $G_k$.
If $v\in V(G_k)\setminus V(H_k)$ has exactly three neighbors
$v_1$, $v_2$ and $v_3$ in $H_k$ and $v_2$ is relaxed, then consider colorings
$\vf_1,\vf_2\in {\cal S}_k$ that differ only in the color of $v_2$.  The coloring $\vf_1$
extends to an $L$-coloring $\psi_1$ of $G_k-vv_2$.  Let $\psi_2$ be obtained from $\psi_1$
by changing the color of $v_2$ to $\vf_2(v_2)$, and note that $\psi_2$ is an $L$-coloring
of $G_k-vv_2$ extending $\vf_2$.  However, either $\psi_1(v)\neq \vf_1(v_2)$ or $\psi_2(v)\neq\vf_2(v_2)$,
hence either $\vf_1$ or $\vf_2$ extends to an $L$-coloring of $G_k$.  This is a contradiction,
since they both belong to ${\cal S}_k$.

Suppose now that $H_k$ has a chord $e=xy$ in $G_k$.  If $G_k=H_k+e$, then since $G$ is not $H$ with a single chord, we have $k>0$.
However, that implies that a vertex of $G_{k-1}$ has degree at most four and a list of size $5$, which is impossible in a critical graph.
It follows that $G_k\neq H_k+e$.  Since $G_k$ is $H_k$-critical, there exists a coloring $\varphi\in {\cal S}_k$
that extends to an $L$-coloring of $H_k+e$, i.e., $\varphi(x)\neq\varphi(y)$.  However,
every coloring in ${\cal S}_k$ extends to every proper subgraph of $G_k$ that includes $H_k$, and it follows that
$\varphi$ extends to an $L$-coloring of $G_k-e$.  This gives an $L$-coloring of $G_k$ extending $\varphi$, contradicting
the assumption that $\varphi\in {\cal S}_k$.  Therefore, we can assume that $H_k$ is an induced cycle in $G_k$.

It follows that a vertex $v\in V(G_k)\setminus V(H_k)$ either has at least four neighbors in $H_k$,
or three neighbors $v_1$, $v_2$ and $v_3$ in $H_k$ such that at most one of the cycles of $H_k+\{v_1v,v_2v,v_3v\}$
distinct from $H_k$ bounds a face.  Then $H_k$ has a $2$-chord $Q$ such that neither of the cycles $K_1$ and $K_2$ of $H_k\cup Q$
distinct from $H_k$ bounds a face.  For $i\in\{1,2\}$, let $G_i'=\cin_{K_i}(G)$.  Suppose first
that it is not possible to choose $Q$ so that neither $G_1'$ nor $G_2'$ is a cycle with one chord.
Since the middle vertex $v$ of $Q$ has degree at least $5$, this can only happen if $V(G_k)\setminus V(H_k)=\{v\}$
and $v$ has degree exactly $5$.  But then $k=0$, since otherwise $G_{k-1}$ would contain a vertex
of degree at most four with a list of size $5$, and we have $\theta(G)=|H|-5+\frac{1}{2|H|+2}<|H|-9/2$.

Finally, suppose that neither $G_1'$ nor $G_2'$ is a cycle with a chord.
By induction and (\ref{eq:Hk theta}), we have $\theta(G)\le \frac{k+1}{2|H|+2}+(|H|-|H_k|)+\theta(G_1')+\theta(G_2')\le 1/2+ (|H|-|H_k|)+|K_1|+|K_2|-9=1/2+ (|H|-|H_k|)+|H_k|-5=|H|-9/2$,
as required.
\end{proof}

To prove Theorem~\ref{thm-connsg}, we need the following simple observation regarding the sizes
of faces in a plane graph.
\begin{lemma}\label{lemma-qsum}
If $H$ is a connected plane graph with $n$ vertices, then
$$\sum_{f\in F(H)} (|f|^2-2)\le 4n^2-8n+2.$$
\end{lemma}
\begin{proof}
We prove the claim by induction on the number of edges of $H$.  If $H$ is a tree,
then it has only one face of length $2n-2$ and the claim follows.
Otherwise, $H$ contains an edge $e$ such that $H-e$ is connected.
Let $f$ be the face of $H-e$ corresponding to two faces $f_1$ and $f_2$
of $H$ separated by $e$.  We have
\begin{eqnarray*}
|f|^2-2&=&(|f_1|+|f_2|-2)^2-2=|f_1|^2+|f_2|^2+2|f_1||f_2|-4|f_1|-4|f_2|+2\\
&\ge& (|f_1|^2-2)+(|f_2|^2-2),
\end{eqnarray*}
since $|f_1|,|f_2|\ge 3$.  Therefore,
$$\sum_{f\in F(H)} (|f|^2-2)\le \sum_{f\in F(H-e)} (|f|^2-2)\le 4n^2-8n+2$$
by the induction hypothesis.
\end{proof}

Theorem~\ref{thm-connsg} is now an easy corollary of Lemma~\ref{lemma-preouf}.

\begin{figure}
\begin{center}
\includegraphics[width=120mm]{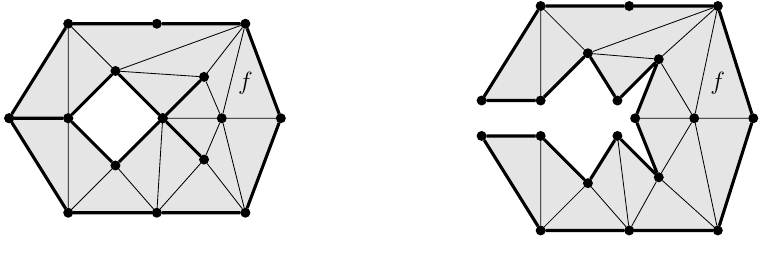}
\end{center}
\caption{Splitting the boundary of a face of $H$.
The boundaries of $f$ and the split cycle $C$ are shown by bold edges.}
\label{fig-splitface}
\end{figure}

\begin{proof}[Proof of Theorem~\ref{thm-connsg}]
Let $F$ be a minimal subgraph of $G$ including $H$ that is not $L$-colorable.
If $F=H$, then the conclusion of Theorem~\ref{thm-connsg} clearly holds.  Hence, assume that $F\neq H$,
and thus $F$ is $H$-critical.
Let $f$ be a face of $H$ and let $F'_f$ be the subgraph of $F$ drawn in the closure of $f$.  In $F'_f$, split the vertices of $f$
so that the interior of $f$ is unchanged and the boundary of $f$ becomes a cycle $C$.
The notion of ``splitting" should be clear from a generic example shown in
Figure \ref{fig-splitface}.  Let $F_f$ be the resulting graph and note that the length of $C$ is $|f|$.
Observe that if $V(F_f)\neq V(C)$, then $F_f$ is $C$-critical, and by Lemma~\ref{lemma-preouf},
\begin{equation}
  |V(F'_f)\setminus V(f)|=|V(F_f)\setminus V(C)|\le (2|f|+2)(|f|-9/2)\le 2(|f|^2-2).
  \label{eq:bound one face F_f}
\end{equation}
Note that the inequality~(\ref{eq:bound one face F_f}) holds when $V(F_f)=V(C)$ as well, since $|f|\ge 3$.
Summing (\ref{eq:bound one face F_f}) over all the faces of $H$, we conclude using Lemma~\ref{lemma-qsum} that $F$ contains at most
$8|V(H)|^2-16|V(H)|+4<8|V(H)|^2-|V(H)|$ vertices not belonging to $H$.  Therefore, $|V(F)|\le 8|V(H)|^2$.
\end{proof}

\subsection{An algorithmic remark}

Lemma~\ref{lemma-preouf} gives rise to a natural algorithm to enumerate all graphs with lists of size 5 critical
with respect to their outer face boundaries.  We proceed inductively by the length $k$ of the cycle $H$; hence
assume that we already know, up to isomorphism, the set ${\cal G}$ of all
planar graphs with precolored outer cycle of length at most $k-1$, such that
the internal vertices have lists of size at least five.  Let ${\cal H}_A$ be all graphs
consisting of a cycle of length $\le k$ with a chord and ${\cal H}_B$ the graphs
consisting of a cycle of length $\le k$ and a vertex with at least three neighbors in the cycle.
Let ${\cal H}_0'$ be the set of all graphs that can be obtained from the graphs in ${\cal H}_A\cup {\cal H}_B$
by pasting the graphs of ${\cal G}$ in some of the faces.  Let ${\cal H}_0$ be the subset of ${\cal H}_0'$
consisting of the graphs that are critical with respect to the boundaries of their outer faces.  For each graph in ${\cal H}_0$, keep adding a vertex
of degree three adjacent to three consecutive vertices of $H$, as long as the resulting graph is critical with
respect to the boundary of its outer face.  This way, we obtain all graphs critical with respect to the boundaries of their outer faces, and these outer faces have length $\ell$.
Lemma~\ref{lemma-preouf} guarantees that this algorithm will finish.  Note also that by omitting ${\cal H}_A$ in the
first step of the algorithm, we can generate such critical graphs whose outer cycle is chordless.

The main difficulty in the implementation is the need to generate all the possible lists in order to
test the criticality, which makes the time complexity impractical.  However, sometimes it
is sufficient to generate a set of graphs that is guaranteed to contain all graphs that are critical
(for some choice of the lists), but may contain some non-critical graphs as well.  To achieve this,
one may replace the criticality testing by a set of simple heuristics that prove that a graph is
not critical. For example, in an $H$-critical graph $G$, each vertex $v\in V(G)\setminus V(H)$ has degree
at least $|L(v)|$, and the vertices whose degrees match the sizes of the lists induce a subgraph $G'$
such that each block of $G'$ is either a complete graph or an odd cycle~\cite{vizing1976}.
There are similar claims forbidding other kinds of subgraphs with specified sizes of lists.  On the positive
side, to prove that a graph is $H$-critical, it is usually sufficient to consider the case that
all lists are equal.  By combining these two tests, we were able to generate graphs critical
with respect to the outer face boundary of length at most $9$.  If the outer face is bounded by an induced cycle,
then there are three of them for length $6$, six for length $7$, $34$ for length $8$ and $182$ for length $9$.
The program that we used can be found at
\texttt{http://atrey.karlin.mff.cuni.cz/\textasciitilde rakdver/5choos/}.

\section{Extending a coloring of a path}
\label{sec-size}

For a path $P$, we let $\ell(P)$ denote its length (the number of its edges).  A vertex of $P$ is an \emph{inside} vertex
if it is not an endvertex of $P$.  The main result of this section follows by using the same basic strategy as in Thomassen's proof of Theorem~\ref{thm-thom} \cite{thomassen1994}.

\begin{theorem}\label{thm-prepathw}
Let $G$ be a plane graph and let $P$ be a subpath of the boundary $H$ of its outer face.  Let $L$ be a list assignment valid
with respect to $P$.  If $G$ is $P$-critical with respect to $L$, then $\omega_{P,L}(G)\le \ell(P)-2$.
\end{theorem}
\begin{proof}
Suppose for a contradiction that $G$ is a counterexample with the smallest number of edges,
and in particular that $\omega_{P,L}(G)\ge \ell-1$, where $\ell=\ell(P)$.
By Theorem~\ref{thm-thom}, we have $\ell\ge 2$.
Furthermore, Theorem~\ref{thm-thom} also implies that if either a vertex or two adjacent vertices
form a vertex-cut $R$ in $G$, then each component of $G-R$ contains a vertex of $P$.
Let $P=p_0p_1\ldots p_{\ell}$.  If $p_i$ is a cut-vertex for some $1\le i\le \ell-1$,
then $G=G_1\cup G_2$, where $G_1,G_2\neq \{p_i\}$ and $G_1\cap G_2=\{p_i\}$.
Let $P_1=P\cap G_1$ and $P_2=P\cap G_2$.  Since $G\neq P$, we can assume that $G_1\neq P_1$.
Note that if $G_2=P_2$, then $\omega_{P_2,L}(G_2)=\ell(P_2)-1$.
If $G_i\neq P_i$, then $G_i$ is $P_i$-critical by Lemma~\ref{lemma-crs}, for $i\in\{1,2\}$.
Furthermore, $p_i$ has weight $1$ in $G$ and weight $0$ both in $G_1$ and $G_2$.  By the minimality of $G$, we have
$\omega_{P,L}(G)=\omega_{P_1,L}(G_1)+\omega_{P_2,L}(G_2)+1\le (\ell(P_1)-2)+(\ell(P_2)-1) + 1 = \ell-2$.
Since $\omega_{P,L}(G)\ge \ell-1$, we conclude that $G$ is $2$-connected.

\begin{figure}
\begin{center}
\includegraphics[width=120mm]{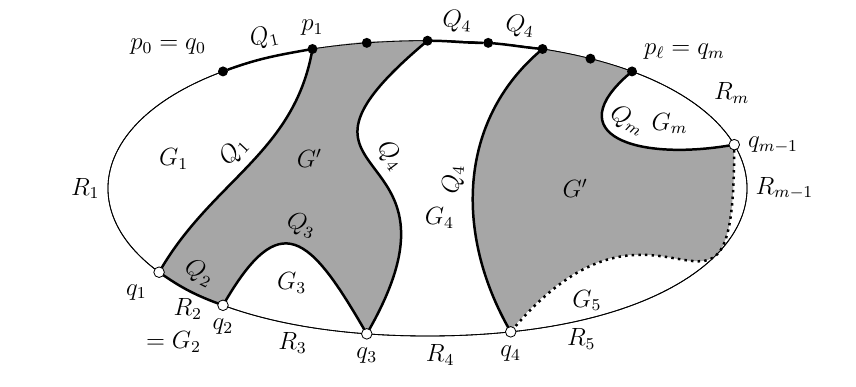}
\end{center}
\caption{Spans in a graph.}
\label{fig-span}
\end{figure}

Suppose that there exists a proper subgraph $G'\supseteq P$ of $G$ and an $L$-coloring $\psi$ of $P$
does not extend to an $L$-coloring of $G'$.  We may choose $G'$ to be $P$-critical.
By the minimality of $G$, we have $\omega_{P,L}(G')\le \ell-2$.
Let $H'$ be boundary of the outer face of $G'$ and let $W'$ be the walk such that the concatenation of
$W'$ and $P$ equals $H'$.  Since $G'$ is $P$-critical, Theorem~\ref{thm-thom}
implies that $W'$ is a path. Let $q_0,\dots,q_m$ be the vertices of $V(H)\cap V(W')$ that are not
internal vertices of the path $P$, listed in the order as they appear in $W'$, where $q_0$ and $q_m$ are the endvertices of $P$. Observe that $q_0,\dots,q_m$ appear in the same order also in $H$.
Each subwalk $Q_i$ of $W'$ from $q_{i-1}$ to $q_i$ ($i=1,\dots,m$) is a span, as defined in Section~\ref{sec-connsg}.
Note that $W'$ is the union of spans $Q_1, \ldots, Q_m$, and each of the
spans is a path. For $1\le i\le m$, let $R_i$ be the segment of $H$ from $q_{i-1}$ to $q_i$, and let $G_i$ be the subgraph of $G$ drawn inside the closed disk bounded by $R_i\cup Q_i$. See Figure~\ref{fig-span}
for an illustration.
%Let $G_i'$ be the subgraph of $G$ such that $G=G_i\cup G_i'$ and $G_i\cap G_i'=Q_i$. Then $P\subseteq G_i'$ and
Note that if $G_i=Q_i$, then $Q_i$ is an edge of $H$.
Observe that $\omega_{G',P,L}(v)\ge 1$
for each internal vertex $v$ of $Q_i$, since $v$ either has a list of size $5$ or it is a cut-vertex in $G'$.
Hence, the total weight in $G'$ of internal vertices of $Q_i$ is at least $\ell(Q_i)-1$.  On the other hand,
their weight in $G$ is $0$.
By the minimality of $G$, we have $\omega_{Q_i,L}(G_i)\le \ell(Q_i)-2$ if $Q_i$ is not equal to an edge of $H$.
If $Q_i$ is an edge of $H$, then $\omega_{Q_i,L}(G_i)=0=\ell(Q_i)-1$.
Consider now an internal face $f$ of $G'$ with boundary $B_f$.  If $f$ is also a face of $G$, then
$\omega_{B_f,L}(\cin_f(G))=|f|-3=\omega_{G',P,L}(f)$.  If $f$ is not a face of $G$, then Lemma~\ref{lemma-preouf} implies that
$\omega_{B_f,L}(\cin_f(G))\le |f|-\frac{9}{2}<|f|-3=\omega_{G',P,L}(f)$.  Since
$$G=G'\cup \bigcup_{i=1}^m G_i\cup \bigcup_{f\in F(G')} \cin_f(G),$$
the weight of $G$ can be obtained by summing the weights of the pieces and adjusting for the different weights of some
vertices in $G$.  More precisely,
\begin{eqnarray*}
  \omega_{P,L}(G) &\le& \omega_{P,L}(G')+\sum_{i=1}^m(\omega_{Q_i,L}(G_i)-(\ell(Q_i)-1)) + \\
   &&    \sum_{f\in F(G')} (\omega_{B_f,L}(\cin_f(G))-\omega_{G',P,L}(f))\\
   &\le& \omega_{P,L}(G') ~\le~ \ell-2.
\end{eqnarray*}
This is a contradiction which proves the following:

\claim{cl-nw-all}{For every proper subgraph $G'$ of $G$, every $L$-coloring $\psi$ of $P$ extends
to an $L$-coloring of $G'$.}

Let $\psi$ be an $L$-coloring of $P$ that does not extend to $G$.
If $L'$ is the list assignment such that $L'(v)=L(v)$ for $v\not\in V(P)$
and $L'(v)=\{\psi(v)\}$ for $v\in V(P)$, \refclaim{cl-nw-all} implies that $G$ is $P$-critical with respect to $L'$.
Note that $\omega_{P,L}(G) = \omega_{P,L'}(G)$ as the sizes of the lists of the vertices of $P$
are not affecting $\omega$.
Consequently, we can assume henceforth that $|L(v)|=1$ for every $v\in V(P)$.
If $V(H)=V(P)$, then by Lemma~\ref{lemma-preouf}, $\omega_{P,L}(G)=\omega_{H,L}(G)\le \ell-2$.
This is a contradiction, hence $p_0$ has a neighbor $w\in V(H)\setminus V(P)$.

If $|L(w)|\ge 4$, then let $L'$ be the list assignment obtained from $L$ by setting $L'(w)=L(w)\setminus L(p_0)$.
Note that $G' = G-p_0w$ is $P$-critical with respect to $L'$, and by the minimality of $G$, $\omega_{P,L'}(G')\le \ell-2$.
Let $f$ be the internal face of $G$ incident with $p_0w$.
Suppose that $u\in V(f)\setminus \{w,p_0\}$.
If $u$ belongs to $V(H)$, then $u$ is a cutvertex in $G'$, and as shown at the beginning of the proof, $u$ is an
internal vertex of $P$.  Therefore, $\omega_{G',P,L'}(u)=1$ and
$\omega_{G,P,L}(u)=0$. On the other hand, if $u\notin V(H)$, then $\omega_{G',P,L'}(u)=2$ and
$\omega_{G,P,L}(u)=0$. Using these facts we obtain a contradiction:
\begin{eqnarray*}
  \omega_{P,L}(G)
   &=& \omega_{P,L'}(G') + \omega_{G,P,L}(f) + 1 -
         \sum_{u\in V(f)\setminus\{w,p_0\}} (\omega_{G',P,L'}(u) - \omega_{G,P,L}(u)) \\
   &\le& \omega_{P,L'}(G') + (|f|-3) + 1 - (|f|-2) = \omega_{P,L'}(G') \le \ell-2.
\end{eqnarray*}

Next, consider the case that $|L(w)|=3$ and $w$ is adjacent to a vertex $p_i$ for some $1\le i\le\ell-1$.
Let $C$ be the cycle composed of $p_0wp_i$ and a subpath of $P$ and let $G'$ be the subgraph of $G$
obtained by removing all vertices and edges of $\cin_C(G)$ except for $p_iw$.
Let $P'=(P\cap G')+p_iw$.  Note that $G'$ is $P'$-critical with respect to $L$.
By the minimality of $G$ and Lemma~\ref{lemma-preouf}, we have
$$\omega_{P,L}(G)=\omega_{P',L}(G')+\omega_{C,L}(\cin_C(G))\le \ell(P')-2+|C|-3=\ell-2.$$
Suppose now that $w$ is adjacent to $p_{\ell}$.  Note that $wp_{\ell}$ is an edge of $H$ and $G\neq H$, hence Lemma~\ref{lemma-preouf} implies that
$\omega_{P,L}(G)=\omega_{H,L}(G)\le |H|-4=\ell-2$.  This is a contradiction.

Finally, suppose that $p_0$ is the only neighbor of $w$ in $P$.  Note that $L(p_0)\subset L(w)$, since
$G$ is $P$-critical. Furthermore, $w$ has only one neighbor $z\in V(H)$ distinct from $p_0$.
Let $S=L(w)\setminus L(p_0)$, $G'=G-w$ and let $L'$ be defined by $L'(v)=L(v)$ if $v$ is not a neighbor of $w$ or if $v=p_0$ or $v=z$, and $L'(v)=L(v)\setminus S$ otherwise.
Since $|S|=2$, $L'$ is a valid list assignment with respect to $P$.
Note that $G'$ is not $L'$-colorable,
as every $L'$-coloring of $G'$ can be extended to an $L$-coloring of $G$ by coloring $w$ using a color
from $S$ different from the color of $z$.  Let $G''$ be a $P$-critical subgraph of $G'$.
Let $Q_1$, \ldots, $Q_m$ be the spans in the boundary of the outer face of $G''$ and let $G_i$ be defined as in the proof
of \refclaim{cl-nw-all}, for $1\le i\le m$, where $w\in V(G_1)$.  The path $Q_1$ is an edge-disjoint
union of paths $M_1$, \ldots, $M_t$, where the endvertices of $M_j$ are neighbors of $w$ and
the internal vertices of $M_j$ are non-adjacent to $w$ for $1\le j\le t$ (with the exception that one of the endvertices
of $M_t$ does not have to be adjacent to $w$).  For $1\le j\le t$, let $C_j$ be the cycle or path formed by
$M_j$ and the edges between $w$ and $M_j$ and let $H_j$ be the subgraph of $G$ split off by $C_j$.
Note that if $v$ is an internal vertex of $M_j$, then $\omega_{G,P,L}(v)=0$ and $\omega_{G'',P,L'}(v)\ge 1$,
while endvertices of $M_j$ have the same weight in $G$ and in $G''$.  Furthermore, $\omega_{G,P,L}(w)=0$. By the minimality of $G$ and Lemma~\ref{lemma-preouf}, we have
$$\omega_{Q_1,L}(G_1)\le \sum_{j=1}^t \omega_{C_j, L}(H_j)\le \sum_{j=1}^t (\ell(M_j)-1).$$
Furthermore, $$\sum_{v\in V(Q_1)} \omega_{G'',P,L'}(v)-\omega_{G,P,L}(v)\ge \sum_{j=1}^t (\ell(M_j)-1)\ge \omega_{Q_1,L}(G_1).$$
We analyse the weights of the other pieces of $G-G'$ in the same way as in the proof of \refclaim{cl-nw-all}
and conclude that $\omega_{P,L}(G)\le\omega_{P,L'}(G'')$.  This contradicts the minimality of $G$
and finishes the proof of Theorem~\ref{thm-prepathw}.
\end{proof}

We need a more precise description of critical graphs in the case that $\ell(P)=2$.
There are infinitely many such graphs, but their structure is relatively simple and it is
described in the sequel.

\begin{figure}
\begin{center}
\includegraphics[width=120mm]{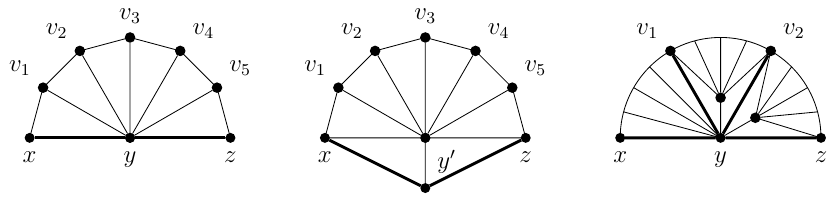}
\end{center}
\caption{A fan, a fat fan, and a fan procession.}
\label{fig-fanprocession}
\end{figure}

For an integer $n\ge 0$, a \emph{fan of order} $n$ with \emph{base} $xyz$ is the graph consisting of the path $xyz$, a path $xv_1\ldots v_nz$ and edges $yv_i$ for $1\le i\le n$.
For an integer $n\ge 1$, a \emph{fat fan of order $n$ with base $xyz$} is the graph consisting of the path $xyz$, a vertex $y'$ adjacent
to $x$, $y$ and $z$, and a fan of order $n$ with base $xy'z$.  A \emph{fan procession with base $xyz$} is
a graph consisting of the path $xyz$, vertices $v_1$, \ldots, $v_{k-1}$ (for some $k\ge 1$) adjacent to $y$, and subgraphs
$G_1$, \ldots, $G_k$ where for $1\le i\le k$, $G_i$ is either a fan or a fat fan with base $v_{i-1}yv_i$ (where we set $v_0=x$ and $v_k=z$).
Each fan or fan procession is a planar near-triangulation, and we consider its unique face of
size $\ge 4$ to be the outer face. See Figure~\ref{fig-fanprocession}.
A fan procession is \emph{even} if all constituent fat fans have even order.
A list assignment $L$ for a fan procession $G$ with base $xyz$ and outer face bounded by $H$ is \emph{dangerous}
if $|L(v)|=3$ for all $v\in V(H)\setminus \{x,y,z\}$ and $|L(v)|=5$ for all $v\in V(G)\setminus V(H)$.

Consider a fat fan $G$ of order $n>0$ with base $xyz$ and a valid list assignment $L$ (with respect to the path $xyz$).  Let $y'$ be the common neighbor of $x$, $y$ and $z$,
and let $v_1v_2\ldots v_n$ be the subpath of the outer face boundary from the definition of a fat fan.
Suppose that $G$ is not $L$-colorable, and let $\vf$ be an $L$-coloring of $xyz$.
It is easy to see that the list assignment $L$ must be dangerous.
Let $S=L(y')\setminus \{\vf(x),\vf(y),\vf(z)\}$.  If there exists $c\in S$ and $1\le i\le n$ such that $c\not\in L(v_i)$,
then $\vf$ extends to an $L$-coloring of $G$ assigning the color $c$ to $y'$.  Therefore, we have $S\subseteq L(v_i)$ for $1\le i\le n$.
Similarly, we have $\vf(x)\in L(v_1)$ and $\vf(z)\in L(v_n)$.  Since $\vf(x)\not\in S$ and
$S\cup \{\vf(x)\}\subseteq L(v_1)$, we have $|S|=2$, $\{\vf(x),\vf(y),\vf(z)\}\subset L(y')$ and $\vf(x)\neq \vf(z)$.
Observe also that $n\ge 2$, as otherwise $y'$ has degree four.  Therefore, $\{\vf(x)\}=L(v_1)\setminus L(v_n)$,
$\{\vf(z)\}=L(v_n)\setminus L(v_1)$ and $\{\vf(y)\}=L(y')\setminus (L(v_1)\cup L(v_n))$.
Therefore, there exists at most one precoloring of $xyz$ that does not extend to an $L$-coloring of $G$.  Furthermore,
if the order $n$ of $G$ is odd, then we can color $y'$ by a color from $S$ and the vertices $v_1$, $v_3$, \ldots, $v_n$
by the other color from $S$ and extend this to an $L$-coloring of $G$.  Therefore, the order of the fat fan $G$ is even.

Using this analysis, it is easy to see that the following holds:

\begin{observation}\label{obs-procprop}
Let $G$ be a fan procession with base $xyz$ and $L$ a dangerous list assignment for $G$. If $\vf_1$ and $\vf_2$ are precolorings of $xyz$ that do not extend to an $L$-coloring
of $G$, and $\vf_1(x)=\vf_2(x)$ and $\vf_1(y)=\vf_2(y)$, then $\vf_1=\vf_2$.  Furthermore,
if there exist two different precolorings of $xyz$ that do not extend to an $L$-coloring of $G$, then $G$ is a fan.
\end{observation}

Conversely, a result of Thomassen~\cite{thom-2007} implies that even fan processions with dangerous list assignments are the only plane graphs with valid list assignments that are $P$-critical for a path $P$ of length two.

\begin{lemma}
\label{lemma-extthom}
Let $G$ be a plane graph, let $H$ be the boundary of the outer face of $G$, and let $P$ a subpath of $H$ of length two.  Let $L$ be a list assignment valid
with respect to $P$.  If $G$ is $P$-critical with respect to $L$, then $G$ is an even fan procession with base $P$ and $L$ is dangerous.
\end{lemma}

\begin{proof}
By Theorem~\ref{thm-prepathw}, $G$ is $2$-connected, all faces other than the outer one are triangles and all vertices in $V(H)\setminus V(P)$ have lists of size three.
Since $G$ is $P$-critical, there exists an $L$-coloring of $P$ that does not extend to an $L$-coloring of $G$.
By Theorem 3 of \cite{thom-2007}, there exists a fan procession $G'\subseteq G$ with base $P$ and $L$ is a dangerous list assignment for $G'$.
By Lemma~\ref{lemma-preouf}, every triangle in $G$ bounds a face.  Furthermore, Theorem~\ref{thm-thom} implies that every chord of $H$ is
incident with the middle vertex of $P$.  We conclude that $G=G'$, and thus $G$ is a fan procession with base $P$.  Furthermore, since an $L$-coloring of $P$
does not extend to an $L$-coloring of $G$, the fan procession is even, as we have argued before.
\end{proof}

\section{Reducing the precolored vertices}
\label{sec-albimp}

One could hope that the proof of Theorem~\ref{thm-prepathw} can be modified to deal with the case
that $G$ contains sufficiently distant precolored vertices.  Most of the inductive applications deal
with the situations which reduce the length of the precolored path, and if the distance between the new
precolored path (one of the spans) from the old one is guaranteed to be bounded by a constant, we could ensure that the
distance between $P$ and the precolored vertices is at least some function of $\ell(P)$.  However, the fact
that there are infinitely many critical graphs makes it difficult to prove such a constraint on
the distance.

To avoid this problem, we restrict ourselves to working with list assignments such that the vertices with
lists of size three form an independent set (almost---for technical reasons, we allow one edge joining vertices
with lists of size three).  In this setting, we easily conclude by combining Theorem~\ref{thm-prepathw} with
Lemma~\ref{lemma-preouf} that the size of critical graphs is bounded.

\begin{lemma}\label{lemma-boundsize}
Let $G$ be a plane graph, let $H$ be the boundary of its outer face, let $P$ be a subpath of $H$ and let $L$ be a list assignment valid
with respect to $P$, such that $G$ contains at most one edge joining vertices with lists of size three.
If $G$ is $P$-critical, then $|V(G)|\le 8\ell(P)^2+4\ell(P)$.
\end{lemma}

\begin{proof}
By induction, we can assume that no cut-vertex belongs to $P$, and thus $G$ is $2$-connected.
The claim is true if $V(G)=V(P)$, thus assume that $V(G)\neq V(P)$.
For $i\in \{3,4,5\}$, let $n_i$ denote the number of vertices with lists of size $i$ in $V(H)\setminus V(P)$.
We have $\omega_{P,L}(G)\ge n_4+2n_5$.
Let $Q$ be a path of length $n_3+2$ whose endvertices coincide with the endvertices of $P$, but is otherwise disjoint from $G$,
and let $G'$ be the graph obtained from $G\cup Q$ by joining each vertex $v\in V(H)\setminus V(P)$ with
$5-|L(v)|$ vertices of $Q$ in the planar way.  Let $L_Q$ be the list assignment to the internal vertices of $Q$ such that
each such vertex has a single color that does not appear in any other list (including the lists of vertices of $G$).
Let $L'$ be the list assignment for $G'$ that matches $L_Q$ on the internal vertices of $Q$ and the list of each vertex
$v\in V(G)\setminus V(P)$ consists of $L(v)$ and the colors of the adjacent internal vertices of $Q$.  Observe that $G'$ is $(P\cup Q)$-critical,
and by Lemma~\ref{lemma-preouf},
$$
   \frac{|V(G)\setminus V(P)|}{2|P\cup Q|+2}=\frac{|V(G')\setminus V(P\cup Q)|}{2|P\cup Q|+2}\le |P\cup Q|-9/2.
$$
Letting $y=|P\cup Q|$, this implies that $|V(G)\setminus V(P)|\le 2y^2-7y-9$, and therefore $|V(G)|\le 2y^2-6y-9$.
Since $L$ is valid, since at most one edge joins vertices with lists of size three,
and since $G$ is $2$-connected, we have $n_3\le n_4+n_5+2$.  Consequently, $\ell(Q)=n_3+2\le n_4+n_5+4\le \omega_{P,L}(G)+4$.
Since $\omega_{P,L}(G)\le \ell(P)-2$ by Theorem~\ref{thm-prepathw}, we have that $y=|P\cup Q|\le 2\ell(P)+2$, and the claim follows.
\end{proof}

Let us remark that a converse of the transformation described in the proof of Lemma~\ref{lemma-boundsize} can be used
to generate all critical graphs satisfying the assumptions of the lemma with the length of $P$ fixed.

Our aim in this section is to show that Lemma~\ref{lemma-albimp} implies a positive answer to Problem~\ref{quest-albertson} (Theorem~\ref{thm-main}).
For technical reasons, we will prove a somewhat convoluted strengthening of this claim, Lemma~\ref{lemma-albcorr} below. In the proof of Theorem~\ref{thm-main},
this Lemma is applied with $p$ being one of the precolored vertices and $X$ consisting of the remaining ones; it gives a conclusion that
that $X=\emptyset$, i.e., only one vertex is precolored, and we can finish the proof using Theorem~\ref{thm-thom} (see the end of the paper
for a precise proof).  We need to introduce several technical definitions.

Let $G$ be a plane graph, let $H$ be the boundary of its outer face, and let $Q$ be a path in $G$.  Suppose that $Q=q_0q_1\ldots q_k$ and $q_0\in V(H)$.
For $0< i < k$, let $L_i$ and $R_i$ be the sets of edges of $G$ incident with $q_i$ such that the cyclic clockwise order (according to the drawing
of $G$ in the plane) of the edges incident with $q_i$ is $q_iq_{i+1}$, $R_i$, $q_iq_{i-1}$, $L_i$.  We define $L_0$ and $R_0$
similarly, except that we consider the outer face instead of the edge $q_iq_{i-1}$ in the order.
We define $G^Q$ as the graph obtained from $G$ by splitting the vertices along $Q$ in
the natural way, i.e., so that $Q$ corresponds to paths $Q_L=q_0^Lq_1^L\ldots q_{k-1}^Lq_k$ and $Q_R=q_0^Rq_1^R\ldots q_{k-1}^Rq_k$
and for $0\le i<k$, the vertex $q_i^L$ is incident with the edges in $L_i$ and the vertex $q_i^R$ is incident with the edges in $R_i$.
If $G$ is given with a list assignment $L$, then let $L^Q$ be the list assignment for $G^Q$ such that $L^Q(q_i^L)=L^Q(q_i^R)=L(q_i)$ for
$0\le i<k$ and $L^Q(v)=L(v)$ for other vertices of $G^Q$.  We say that $G^Q$ and $L^Q$ are \emph{obtained by cutting along $Q$}.

For integers $M$ and $k$, let $D(M, k)=M+2$ if $k\le 1$ and $D(M, k)=D(M, k-1)+16k^2+8k$ if $k\ge 2$. Note that there is a simple explicit formula for the values $D(M,k)$, but we shall only use its recursive description.
A set $X$ of vertices is \emph{$M$-scattered} if the distance between any two elements of $X$ is
at least $\max\{D(M, 2M+11), D(M,2)+D(M,6)+1\}$.

Let $Q=q_0q_1\dots q_k$ be a path of length $k$. If $k$ is even, then $q_{k/2}$ is said to be the \emph{central vertex} of $Q$; if $k$ is odd, then each of the two vertices $q_{(k-1)/2}$ and $q_{(k+1)/2}$ is a \emph{central vertex} of $Q$.

\begin{lemma}
\label{lemma-albcorr}
Suppose that there is a positive integer $M$ such that the conclusion of Lemma~\ref{lemma-albimp} holds.
Let $G$ be a plane graph, let $P$ be a subpath of the boundary $H$ of its outer face and let $p$ be a central vertex of $P$.
Let $X$ be an $M$-scattered subset of\/ $V(G)$ such that the distance between $p$ and $X$ is at least $D(M,\ell(P))$.
Let $L$ be a list assignment for $G$ that is $M$-valid with respect to $P$ and $X$.
Furthermore, assume that there is at most one edge $uv\in E(G)$ such that $u,v\in V(G)\setminus V(P)$ and $|L(u)|=|L(v)|=3$,
and if such an edge exists, then $\ell(P)\le 1$, $u$ or $v$ is adjacent to $p$ and the distance between $p$ and $X$ is at least $D(M,2)-1$.
If $G$ is $P$-critical with respect to $L$, then $X=\emptyset$.
\end{lemma}

\begin{proof}
For a contradiction, suppose that $G$ is a counterexample to Lemma~\ref{lemma-albcorr} with the smallest number of edges
that do not belong to $P$, subject to that, with the smallest number of vertices, and subject to that, with the
largest number of vertices in $P$.
Since $G$ is $P$-critical, every vertex $v\in V(G)\setminus V(P)$ satisfies $\deg(v)\ge |L(v)|$.
Let $\ell = \ell(P)$ and $P=p_0p_1\ldots p_{\ell}$. If $\ell$ is odd, choose the labels so that $p=p_{(\ell + 1)/2}$.

Suppose that $G$ is disconnected.  Since $G$ is $P$-critical, it has two components: one is equal to $P$ and
the other one, $G'$, is not $L$-colorable.   Choose $v\in V(H)\cap V(G')$ arbitrarily, and let $P'$ be the path consisting of $v$.
Note that $G'$ is $P'$-critical.  If $G'$ with the path $P'$ satisfies the assumptions of Lemma~\ref{lemma-albcorr}, then by
the minimality of $G$ we have $X\cap V(G')=\emptyset$, and thus $X=\emptyset$.  This is a contradiction, and thus
the distance from $v$ to the closest vertex $x\in X$ is at most $M+1$.
Let $Q$ be the shortest path between $v$ and $x$ and let $G^Q$, $Q_L$ and $Q_R$ with the list assignment $L^Q$
be obtained from $G'$ by cutting along $Q$.  Let $Q'=Q_L\cup Q_R$ and $X'=X\setminus \{x\}$.
Note that $x$ is the central vertex of $Q'$ and its distance to any $u\in X'$ is at least $D(M,\ell(Q'))$,
since $X$ is $M$-scattered and $\ell(Q')\le 2M+2$.
In particular, $L^Q$ is $M$-valid with respect to $Q'$ and $X'$. Furthermore, $G^Q$ is $Q'$-critical with respect to $L^Q$.
To see this, consider an arbitrary edge $e\in E(G')\setminus E(Q)$. Since $G$ is $P$-critical, there exists an $L$-coloring of $P$ that extends to $G-e$ but not to $G$. The coloring of $G-e$ induced on $Q$ gives rise to an $L^Q$-coloring of $Q'$ that extends to $G^Q-e$ but not to $G^Q$. This shows that
$G^Q$ is $Q'$-critical. Since the distances in $G^Q$ are not shorter than those in $G$, the graph $G^Q$ satisfies all assumptions of Lemma~\ref{lemma-albcorr}.
By the minimality of $G$, we conclude that $X'=\emptyset$.  But then $|X|=1$ and $G'$ (with no precolored path) contradicts Lemma~\ref{lemma-albimp}.

Therefore, $G$ is connected.  In particular, if $\ell=0$, then we can include another vertex of $H$ in $P$; therefore, we can assume that $\ell\ge 1$.
Since $G$ is connected, its outer face has a facial walk, which we write as
$p_{\ell}\ldots p_1p_0v_1v_2v_3\ldots v_s$.

Suppose that the distance between $P$ and $X$ is at most $M+5$. Then the distance from $p$ to $X$ is at most $M+\ell+5$.  By the assumptions of the lemma, this distance is at least $D(M,\ell)$, which is only possible if $\ell\le 1$. As assumed above, this means that $\ell=1$. Moreover, the assumptions of the lemma imply that no two vertices with lists of size three are adjacent.  Let $Q$ be a shortest path between $P$ and
a vertex $x\in X$.  Let $G^Q$, $Q_L$ and $Q_R$ with the list assignment $L^Q$ be obtained from $G$ by cutting along $Q$.
Let $Q'$ be the path consisting of $Q_L\cup Q_R$ and of the edge of $P$, and let $X'=X\setminus \{x\}$. Note that $\ell(Q') \le 2M+11$. Since $X$ is $M$-scattered, so is $X'$, and the distance in $G^Q$ from the central vertex $x$ of $Q'$ to $X'$ is at least $D(M,2M+11) \ge D(M,\ell(Q'))$.
As in the previous paragraph, we conclude that since $G^Q$ is $Q'$-critical with respect to $L^Q$, we have $X'=\emptyset$. Then $|X|=1$ and, consequently, $G$ contradicts the postulated property of the constant $M$. Therefore, we conclude:

\claim{cl-farle2}{The distance between $P$ and $X$ is at least $M+6$.}

We say that a cycle $T$ in $G$ is \emph{separating} if $V(\cin_T(G))\ne V(T)$ and $T$ does not bound the outer face of $G$.
Let $T=t_1\ldots t_k$ be a separating $k$-cycle in $G$, where $k\le 4$. Suppose that the distance from $t_1$ to $P$ is at most $6-k$.
Let us choose such a cycle with $\cin_T(G)$ minimal; it follows that $T$ is an induced cycle.  By Lemma~\ref{lemma-crs}, $\cin_T(G)$ is $T$-critical, and
thus there exists an $L$-coloring $\psi$ of $T$ that does not extend to an $L$-coloring of $\cin_T(G)$.
Let $G'=\cin_T(G)-\{t_3,\ldots, t_k\}$.  Let $L'$ be the list assignment for $G'$ such that $L'(t_1)=\{\psi(t_1)\}$, $L'(t_2)=\{\psi(t_2)\}$
and $L'(v)=L(v)\setminus\{\psi(t_i)\mid vt_i\in E(G), 3\le i\le k\}$ for other vertices $v\in V(G')$.
Note that no two vertices with lists of size three are adjacent in $G'$, as otherwise we have $k=4$ and $t_3t_4$ is incident with a separating
triangle contradicting the choice of $T$.  The graph $G'$ is not $L'$-colorable,
hence it contains a $t_1t_2$-critical subgraph $G''$.  By \refclaim{cl-farle2}, $L'$ is an $M$-valid list
assignment for $G''$, and the distance between $t_1$ and $X\cap V(G')$ is at least $M+2$.
By the minimality of $G$, it follows that $X\cap V(G'')=\emptyset$.  However, then $G''$ contradicts
Theorem~\ref{thm-thom}.  We conclude that the following holds:

\claim{cl-notr}{If $T\neq H$ is a separating $k$-cycle in $G$, where $k\le4$, then the distance
between\/ $T$ and\/ $P$ is at least\/ $7-k$.}

Similarly, by applying induction, we obtain the following property.

\claim{cl-nocut}{If $R$ is either a chord of $H$ that
does not contain an internal vertex of $P$, or $R$ is a cut-vertex of $G$, then the distance between $R$ and $P$ is
at least~$4$.}

\begin{proof}
Suppose first that $R$ does not contain an internal vertex of $P$.
Let $G'$ be the subgraph of $G$ split off by $R$.  By Lemma~\ref{lemma-crs}, $G'$ is
$R$-critical, and Theorem~\ref{thm-thom} implies that $X\cap V(G')\neq\emptyset$.
If the distance from $P$ to $R$ is at most 3, then by \refclaim{cl-farle2}, the
distance between $R$ and $X$ is at least $M+2=D(M,\ell(R))$.  If $G'-V(R)$ does not contain two adjacent vertices with lists of size three, this
contradicts the minimality of $G$.  If $uv\in E(G'-V(R))$ and $|L(u)|=|L(v)|=3$, then by the assumptions, we have
$\ell=1$ and $u$ or $v$ is adjacent to $p$.  Consequently, $p\in V(R)$, and thus the distance between a central vertex $p$ of
$R$ and $X$ is at least $D(M,2)-1$.  Again, we have a contradiction with the minimality of $G$.

Suppose now that $P$ contains a cut-vertex $v$ of $G$, and let $G_1$ and $G_2$ be the two maximal connected subgraphs
of $G$ that intersect in $v$.  For $i\in \{1,2\}$, let $P_i=P\cap G_i$ and note that either $P_i=G_i$
or $G_i$ is $P_i$-critical by Lemma~\ref{lemma-crs}.  By the minimality of $G$, we conclude that
neither $G_1$ nor $G_2$ contains a vertex of $X$, and thus $X=\emptyset$.  This contradiction completes the proof.
\end{proof}

Next, we claim the following.

\claim{cl-sepnox}{Let $C\subset G$ be a cycle of length at most\/ $\ell+1$ such that $C\neq H$ and the distance between
$C$ and $p$ is at most\/ $8\ell^2+4\ell$.  Then $\cin_C(G)$ contains no vertices of $X$.}

\begin{proof}
The length of $C$ is at least three, and thus $\ell\ge 2$.
If $x\in X$ belongs to $C$, then the distance from $x$ to $p$ is less than $8\ell^2 + 5\ell<D(M,\ell)$, a contradiction.
Thus, we may assume that $V(C)\cap X = \emptyset$ and, in particular, that $C$ does not bound a face.
If $\ell(C)\le \ell$, then the claim holds even under a weaker assumption that the distance between
$C$ and $P$ is at most $16\ell^2+8\ell$.  Indeed, consider a spanning subpath $Q$ of $C$ of length $\ell(C)-1$ such that
the distance between $p$ and a central vertex $q$ of $Q$ is at most $16\ell^2+8\ell$.  The distance
of every vertex of $X$ in $\cin_C(G)$ from $q$ is at least $D(M,\ell)-16\ell^2-8\ell\ge D(M,\ell(Q))$.
By Lemma~\ref{lemma-crs}, we have that $\cin_C(G)$ is $Q$-critical, and the claim follows by the minimality of $G$.

Suppose now that $\ell(C)=\ell+1$ and let $C=c_0c_1\ldots c_{\ell}$, where $c_{\lceil \ell/2\rceil}$ is the vertex of $C$
nearest to $p$.  There exists an $L$-coloring $\vf$ of $C$ that
does not extend to an $L$-coloring of $\cin_C(G)$.  Let $d$ be a new color that does not appear in any of the lists
and let $L'$ be the list assignment obtained from $L$ by replacing $\vf(c_{\ell})$ by $d$ in the lists of $c_{\ell}$ and
its neighbors and by setting $L'(c_0)=\{\vf(c_0),\vf(c_1),d\}$.  Let $\vf'$ be the coloring of the path $C'=c_1c_2\ldots c_{\ell}$
such that $\vf'(c_{\ell})=d$ and $\vf'$ matches $\vf$ on the rest of the vertices.  The coloring $\vf'$ does not
extend to an $L'$-coloring of $\cin_C(G)$; hence, $\cin_C(G)$ contains a subgraph $F\supset C'$ that is $C'$-critical with respect to $L'$.
The distance of $X\cap V(F)$ from the central vertex $c_{\lceil \ell/2\rceil}$ of $C'$ is at least
$D(M,\ell)-8\ell^2-4\ell>D(M,\ell(C'))$.
By the minimality of $G$, we conclude that $F$ contains no vertex of $X$.  By Theorem~\ref{thm-prepathw}, we have
$\omega_{C',L'}(F)\le \ell-3$, and in particular, every face of $F$ has length at most $\ell$.
By Lemma~\ref{lemma-boundsize}, the distance from $c_{\lceil \ell/2\rceil}$ to every vertex of $F$ is less than $8\ell^2+4\ell$, thus the distance between
every vertex of $F$ and $p$ is at most $16\ell^2+8\ell$.
By the previous paragraph, we conclude that no vertex of $X$ appears in the interior of any face of $F$.
Let $Q$ be the path in the boundary of the outer face of $F$, distinct from $C'$, joining $c_1$ with $c_{\ell}$.  If $v\neq c_0$ is an internal vertex
of $Q$, then $\omega_{F,C',L'}(v)\ge 1$, hence $Q$ contains at most $\ell-3$ such internal vertices.  It follows that
$Q+c_1c_0c_{\ell}$ is either a cycle of length at most $\ell$ (if $c_0\not\in V(Q)$) or a union of two cycles of total length
at most $\ell+1$ (if $c_0\in V(Q)$).  In both cases, the interiors of the cycles do not contain any vertex of $X$ by the previous paragraph.
Consequently, $X\cap V(\cin_C(G))=\emptyset$ as claimed.
\end{proof}

Let $\psi$ be an $L$-coloring of $P$ that does not extend to an $L$-coloring of $G$.
We are going to show that $\psi$ extends to all proper subgraphs of $G$ that contain $P$.
To prove this, we use the following claim, which we formulate in greater generality
for future use.

\claim{cl-empty}{Let $G'$ be a proper subgraph of $G$ with $P\subseteq G'$, and let $H'$ be the boundary of
the outer face of $G'$.  Let $L'$ be a list assignment for $G'$ that is $M$-valid with respect to $P$ and $X\cap V(G')$,
such that $L'(v)=L(v)$ for all $v\in V(G')\setminus V(H')$ and $L'(v)\subseteq L(v)$ for all $v\in V(H')$.
Furthermore, assume that there is at most one edge $u'v'\in E(G')$ such that $u',v'\in V(G')\setminus V(P)$ and $|L'(u)|=|L'(v)|=3$,
and if such an edge exists, then $\ell(P)\le 1$, $u'$ or $v'$ is adjacent to $p$ and the distance between $p$ and $X\cap V(G')$ is at least $D(M,2)-1$.
Suppose that $G'$ is $P$-critical with respect to $L'$.  Then $X\cap V(G')=\emptyset$ and every internal face $f$ of $G'$ satisfies $X\cap V(\cin_f(G))=\emptyset$.
Furthermore, suppose that $Q$ is a span\footnote{As defined in Section~\ref{sec-connsg}.} in $G$ (not necessarily contained in $G'$) with a central vertex in $G'$.
If $\ell(Q)\le \ell-1$, then the subgraph of $G$ split off by $Q$ contains no vertices of $X$.  In particular, this is the case if $Q\subset H'$ and
$L'(v)=L(v)$ for all internal vertices $v$ of $Q$.}
\begin{proof}
Since $G'$ satisfies the assumptions of Lemma~\ref{lemma-albcorr}, the minimality of $G$ implies that
$X\cap V(G')=\emptyset$.  By Theorem~\ref{thm-prepathw}, it follows that
$\omega_{P,L'}(G')\le \ell-2$, and in particular, $\ell\ge 2$.  Let $f$ be a face of $G'$ distinct from the outer one
such that $\cin_f(G)\neq f$.  Since $\omega(f)\le \ell-2$, we have $\ell(f)\le \ell+1$.
Furthermore, by Lemma~\ref{lemma-boundsize}, the distance in $G'$ between $f$ and $p$ is
at most $8\ell^2+4\ell$.  By \refclaim{cl-sepnox}, no vertex of $X$ appears in $\cin_f(G)$.

Consider now a span $Q$ of $G$ with a central vertex $q$ in $G'$, and let $G_Q$ be the $Q$-component of $G$ split off by $Q$.
By Lemma~\ref{lemma-boundsize}, the distance between $p$ and $q$ in $G'$ is at most $8\ell^2+4\ell$, and thus if
$\ell(Q)\le \ell-1$, then the distance between $q$ and $X$ in $G_Q$ is at least $D(M,\ell)-8\ell^2-4\ell\ge D(M,\ell(Q))$.
Observe that $G_Q$ is $Q$-critical if $G_Q\ne Q$, and by the minimality of $G$, $G_Q$ contains no vertices of $X$.

Suppose that $Q$ happens to be a subpath of $H'$ such that $L'(v)=L(v)$ for all internal vertices $v$ of $Q$, implying $\omega_{G',P,L'}(v)\ge 1$.
Then, $\ell(Q)\le \omega_{P,L'}(G')+1\le \ell-1$, and the argument of the previous paragraph applies.
\end{proof}

Suppose that there exists a proper subgraph $G'\subset G$ such that $P\subset G'$ and
$\psi$ cannot be extended to an $L$-coloring of $G'$.  Let $G'$ be minimal subject to this
property.  Then $G'$ is a $P$-critical graph.  If $G'$ does not satisfy the assumptions of \refclaim{cl-empty} (with the list assignment $L$),
then $\ell=1$ and there exist adjacent vertices $u,v\in V(G')\setminus V(P)$ with lists of size three such that neither
of them is adjacent to $p_1$ in $G'$, while $u$ is adjacent to $p_1$ in $G$.
Let $c$ be a new color that does not appear in any of the lists.  Let $L'$ be the list assignment for $G'$
obtained from $L$ by replacing $\psi(p_1)$ by $c$ in the lists of all vertices adjacent to $p_1$ in $G'$
and by setting $L'(p_0)=\{\psi(p_0)\}$, $L'(p_1)=\{c\}$, and $L'(u)=L(u)\cup \{c\}$.  Each $L'$-coloring of $G'+up_1$
corresponds to an $L$-coloring of $G'$ extending $\psi$, hence $G'+up_1$ is not $L'$-colorable and it contains a $P$-critical subgraph $G''$.
Note that $|L'(u)|=4$ and hence no two vertices with lists of size three are adjacent in $G''$.
However, the minimality of $G$ implies that $G''$ contains no vertices of $X$, and we obtain a contradiction with Theorem~\ref{thm-thom}.

Hence, \refclaim{cl-empty} applies to $G'$, showing that no vertex of $X$ is contained in $G'$, in $\cin_f(G)$ for internal faces $f$ of $G'$,
or in the subgraphs of $G$ split off by the spans contained in the boundary of the outer face of $G'$.
We conclude that $X=\emptyset$.  This is a contradiction; therefore, $\psi$ extends to all proper
subgraphs of $G$ that contain $P$.  Equivalently, we can restrict the lists of the vertices of $P$ to the singleton lists
given by $\psi$, and $G$ is $P$-critical with respect to the resulting list assignment.  Hence, we can assume the following.

\claim{cl-colne}{The vertices of $P$ have lists of size one, $G$ is not $L$-colorable and every proper subgraph of $G$ that contains $P$ is $L$-colorable.}

Let us fix $\psi$ as the unique $L$-coloring of $P$.

Consider a chord $e=uv$ of $H$ at distance at most three from $P$.
By \refclaim{cl-nocut}, we can assume that $u$ is an internal vertex of $P$, and in particular $\ell\ge 2$.
If $v$ belonged to $P$ as well, then by \refclaim{cl-colne} we have $G=P+e$, implying $X=\emptyset$.
Hence, $v$ does not belong to $P$.

Let $G_1$ and $G_2$ be the maximal connected subgraphs of $G$
intersecting in $uv$, such that $G_1\cup G_2 = G$, $p_\ell\in V(G_1)$, and $p_0\in V(G_2)$. Let $P_i=(P\cap G_i)+e$.  For $i\in \{1,2\}$,
Lemma~\ref{lemma-crs} implies that the graph $G_i$ is $P_i$-critical.
Note that either $\ell(P_i)<\ell(P)$, or $\ell(P_i)=\ell$ and $p$ is a central vertex of $P_i$.
We conclude that the distance between a central vertex of $P_i$ and $X$ is at least $D(M, \ell(P_i))$.
By the minimality of $G$, we have $X\cap V(G_i)=\emptyset$
for $i\in\{1,2\}$.  It follows that $X=\emptyset$, which is a contradiction.
Therefore, we have:

\claim{cl-chord}{The distance of any chord of $H$ from $P$ is at least four.}

In particular, if $\ell\ge 2$, then $p$ is not incident with a chord, and thus $G$ cannot contain adjacent vertices
with lists of size three.

\claim{cl-no33ifg2}{If $\ell\ge 2$, then no edge $uv\in E(G)$ satisfies $|L(u)|=|L(v)|=3$.}

Let $s=|V(H)\setminus V(P)|$. A consequence of \refclaim{cl-colne} is that $s\ge1$ (if $s$ were equal to $0$, then $G-p_0p_\ell$ would contradict the claim). We can say more:

\claim{cl-no44}{If\/ $|L(v_1)|>3$, then $|L(v_1)|=4$, $s\ge 2$ and $|L(v_2)|=3$.}

Otherwise, suppose that $|L(v_1)|=5$, or $|L(v_1)|=4$ and either $s=1$ or $|L(v_2)|\ge 4$.
Let $G'=G-p_0v_1$ and let $L'$ be the list assignment obtained from $L$ by removing $\psi(p_0)$ from
the list of $v_1$.  The assumptions together with \refclaim{cl-chord} imply that if $|L'(v_1)|=3$, then
$v_1$ is not adjacent to any vertex with list of size three in $G'$.  By \refclaim{cl-colne}, $G'$ is $P$-critical with respect to $L'$, contradicting the minimality of~$G$.

Suppose now that $\ell\ge 2$ and a vertex $v$ is adjacent to $p_0$, $p_1$ and $p_2$.  By \refclaim{cl-chord},
we have $v\not\in V(H)$.  Let $P'=p_0vp_2p_3\ldots p_{\ell}$,
$H'=p_0vp_2\ldots p_{\ell}v_s\ldots v_1$ and $G'=\cin_{H'}(G)$.  By Lemma~\ref{lemma-crs},
$G'$ is $P'$-critical.  If $\ell\ge 3$, then $p$ is a central vertex of $P'$ and by the minimality of $G$, we have $X\cap V(G')=\emptyset$.
Furthermore, \refclaim{cl-notr} implies that $p_0p_1v$ and $p_1p_2v$ bound faces of $G$, and thus
$X=\emptyset$.  This contradiction shows the following.

\claim{cl-ninp}{If\/ $\ell\ge2$ and $p_0$, $p_1$ and $p_2$ have a common neighbor, then $\ell=2$.}

For a vertex $v\in V(G)\setminus V(P)$, let us define the reduced list $S(v)$ by
$$
   S(v)=L(v)\setminus \{\psi(r) : r\in V(P), vr\in E(G)\}.
$$

\claim{cl-new_claim_11}{If $v$ is a vertex of $V(G)\setminus V(P)$ with $k$ neighbors in $P$, then $|S(v)|=|L(v)|-k$.}

To see this, suppose $v$ is adjacent to a vertex $r\in V(P)$ and $\psi(r)\notin L(v)$, or $v$ is adjacent to two vertices $r,r'\in V(P)$ with $\psi(r)=\psi(r')$. Then we can remove the edge $vr$ and obtain a contradiction to the last assertion in \refclaim{cl-colne}.

Consider a nonempty set $Y\subseteq V(G)\setminus V(P)$ and a partial coloring $\vf$ of the subgraph of $G$ induced by $Y$ from the reduced list assignment $S$.  The domain of this partial coloring is denoted by $\dom(\vf)\subseteq Y$.
We define $L_{\vf}$ as the list assignment such that
$$L_{\vf}(z)=L(z)\setminus \{\vf(y):y\in\dom(\vf),yz\in E(G),\vf(y)\in S(z)\}$$
for every $z\in V(G-Y)$.

\begin{figure}
\begin{center}
\includegraphics[width=120mm]{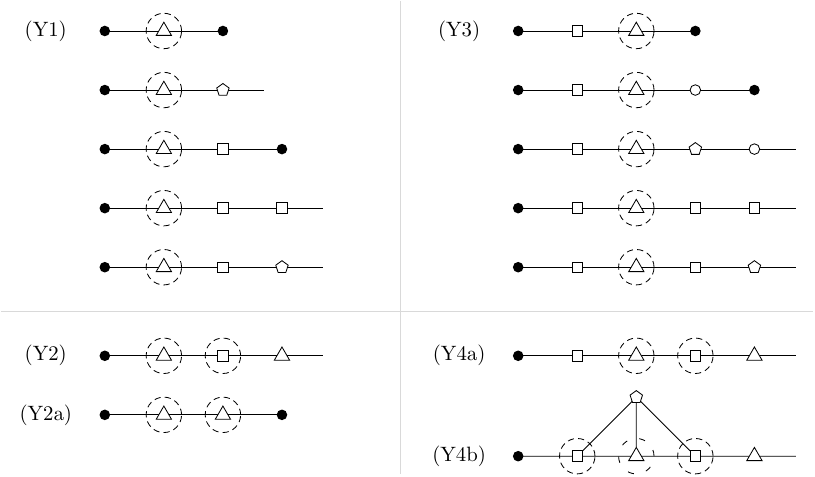}
\end{center}
\caption{The definition of the set $Y$. Full-circle vertices belong to $P$, numbers of sides of polygonal vertices indicate sizes of the lists.
Empty-circle vertices do not belong to $P$, but may have any lists.}
\label{fig-Y}
\end{figure}

We now define a particular set $Y\subseteq V(H)\setminus V(P)$ (see Figure~\ref{fig-Y} for reference) and a partial $L$-coloring $\vf$ of $Y$ as follows.
\begin{definition}\label{def-y}\leavevmode
\begin{itemize}
\item[(Y1)] If $|L(v_1)|=3$ and one of the following holds:
\begin{itemize}
\item[$\circ$] $s=1$, or
\item[$\circ$] $s\ge 2$ and $|L(v_2)|=5$, or
\item[$\circ$] $s=2$ and $|L(v_2)|=4$, or
\item[$\circ$] $s\ge 3$, $|L(v_2)|=4$ and $|L(v_3)|\neq 3$,
\end{itemize}
then $Y=\{v_1\}$ and $\vf(v_1)\in S(v_1)$ is chosen arbitrarily.
\item[(Y2)] If $|L(v_1)|=3$, $s\ge 3$, $|L(v_2)|=4$ and $|L(v_3)|=3$, then $Y=\{v_1,v_2\}$ and $\vf$ is chosen so that $\vf(v_2)\in L(v_2)\setminus L(v_3)$
and $\vf(v_1)\in S(v_1)\setminus\{\vf(v_2)\}$.
\item[(Y2a)] If $s=2$ and $|L(v_1)|=|L(v_2)|=3$, then $Y=\{v_1,v_2\}$ and $\vf(v_1)\in S(v_1)$ and $\vf(v_2)\in S(v_2)$ are chosen arbitrarily so that $\vf(v_1)\neq \vf(v_2)$.
\item[(Y3)] If $|L(v_1)|=4$, $s\ge 2$, $|L(v_2)|=3$, and one of the following holds:
\begin{itemize}
\item[$\circ$] $s\le 3$, or
\item[$\circ$] $s\ge 4$ and $|L(v_3)|=5$, or
\item[$\circ$] $s\ge 4$, $|L(v_3)|=4$ and $|L(v_4)|\neq 3$,
\end{itemize}
then $Y=\{v_2\}$.  If $s=3$ and $|L(v_3)|=3$, then $\vf(v_2)$ is chosen in $L(v_2)\setminus S(v_3)$,
otherwise $\vf(v_2)\in S(v_2)$ is chosen arbitrarily.
\item[(Y4)] If $s\ge 4$, $|L(v_1)|=4$, $|L(v_2)|=3$, $|L(v_3)|=4$ and $|L(v_4)|=3$, then:
\begin{itemize}
\item[(Y4a)] If $v_1$, $v_2$ and $v_3$ do not have a common neighbor, then $Y=\{v_2, v_3\}$ and $\vf$ is chosen
so that $\vf(v_3)\in L(v_3)\setminus L(v_4)$ and $\vf(v_2)\in L(v_2)\setminus \{\vf(v_3)\}$.
\item[(Y4b)] If $v_1$, $v_2$ and $v_3$ have a common neighbor, then $Y=\{v_1, v_2, v_3\}$ and $\vf$ is chosen
so that $\vf(v_3)\in L(v_3)\setminus L(v_4)$, $\vf(v_1)\in S(v_1)$ and either at least one of $\vf(v_1)$ and $\vf(v_3)$
does not belong to $L(v_2)$, or $\vf(v_1)=\vf(v_3)$.  The vertex $v_2$ is left uncolored.
Note that this is the only case where $\dom(\vf)\ne Y$.
\end{itemize}
\end{itemize}
\end{definition}

By using \refclaim{cl-no44} (together with \refclaim{cl-chord} and the condition on adjacent vertices with lists of size 3) it is easy to see that $Y$ and $\vf$ are always defined.
(Note that in the case of adjacent vertices $u,v$ with lists of size 3, we can assume that $u$ is adjacent to $p$ and $\ell=1$. Let us recall that
if $\ell=1$, then we have chosen $p=p_1$; hence, $u=v_s$ and $v=v_{s-1}$. Therefore, only (Y2a) and (Y3) are needed to deal with this special case.)
We remark that the following is true.

\claim{cl-coloring_extends}{Every $L_{\vf}$-coloring of\/ $G-Y$ extends to an $L$-coloring of $G$.}

Indeed, this is obviously true if $\dom(\vf)=Y$. The only case when $\dom(\vf)\ne Y$ is (Y4b), where $Y=\{v_1,v_2,v_3\}$ and $\dom(\vf)=\{v_1,v_3\}$. However, $\deg(v_2)=3$ by \refclaim{cl-notr}, and $|L_{\vf}(v_2)|\ge2$ by the choice of $\vf(v_1)$ and $\vf(v_3)$. This implies that any $L_{\vf}$-coloring of $G-Y$ extends to $v_2$ and proves \refclaim{cl-coloring_extends}.
Consequently, $G-Y$ is not $L_{\vf}$-colorable. We let $G_{\vf}$ be a $P$-critical subgraph of $G-Y$.

Using \refclaim{cl-farle2} and \refclaim{cl-chord}, it is easy to verify that the choice of $Y$ and $\vf$ ensures that the list assignment $L_{\vf}$ for $G-Y$
is $M$-valid with respect to $P$ and $X$.  Let us now distinguish two cases depending on whether $G_{\vf}$
contains two adjacent vertices with lists of size three (that did not have lists of size three in $G$ as well) or not.

\begin{itemize}
\item  Suppose first that \textbf{no two vertices $u, v\in V(G_{\vf})$ such that $|L_{\vf}(u)|=|L_{\vf}(v)|=3$ and $\max(|L(u)|,|L(v)|)>3$ are adjacent.}
If $G_{\vf}$ with the list assignment $L_{\vf}$ does not satisfy the assumptions of Lemma~\ref{lemma-albcorr}, this is only because there are adjacent vertices with lists of size 3 that are no longer adjacent to $p$ in $G_{\vf}$. More precisely, in that case
$\ell=1$, $|L(v_s)|=|L(v_{s-1})|=3$, $v_sv_{s-1}\in E(G_{\vf})$ and $p_1v_s\not\in E(G_{\vf})$.
Let $c$ be a new color that does not appear in any of the lists and let $L'$ be the list assignment obtained
from $L_{\vf}$ by replacing $\psi(p_1)$ with $c$ in the lists of vertices adjacent to $p_1$ in $G_{\vf}$
and by setting $L'(p_1)=\{c\}$ and $L'(v_s)=L(v_s)\cup\{c\}$.  Observe that $G_{\vf}+p_1v_s$ is not $L'$-colorable
and thus it contains a $P$-critical subgraph $G'$.  By the minimality of $G$, we have $X\cap V(G')=\emptyset$.
However, then $G'$ contradicts Theorem~\ref{thm-thom}.

Therefore, $G_{\vf}$ with the list assignment $L_{\vf}$ satisfies the assumptions of Lemma~\ref{lemma-albcorr}.
By the minimality of $G$, we conclude that $G_{\vf}$ does not contain any vertex of $X$.
By Theorem~\ref{thm-prepathw}, we have $\ell\ge 2$ and $\omega_{P,L_{\vf}}(G_{\vf})\le \ell-2$.
Let $Q$ be the span contained in the boundary of the outer face of $G_{\vf}$ such that the $Q$-component $G_Q$ split off by $Q$ contains $Y$.
By \refclaim{cl-empty}, if $f$ is a face of $G_{\vf}$, then $\cin_f(G)$ contains no vertex of $X$, and if $Q'$ is a span different from $Q$, then
the subgraph of $G$ split off by $Q'$ contains no vertex of $X$.  Since $X\neq\emptyset$, it follows that
$G_Q$ contains a vertex of $X$.  Also, \refclaim{cl-empty} implies that $\ell(Q)\ge \ell$.

If $v$ is an internal vertex of $Q$, then $\omega_{P,L_{\vf}}(v)\ge 1$,
unless $|L_{\vf}(v)|=3$.  Since the sum of the weights of the internal vertices of $Q$ is at most $\omega_{P,L_{\vf}}(G_{\vf})\le \ell-2$,
we conclude that at least one internal vertex of $Q$ has a list of size three.  This is only possible in the cases (Y2), (Y4a), and (Y4b);
\refclaim{cl-no33ifg2} excludes the case (Y2a).  Furthermore,
observe that only one internal vertex of $Q$ has a list of size three by \refclaim{cl-notr}; let $v$ denote this vertex.
It also follows that $\ell(Q)=\ell$ and that all internal vertices of $Q$ other than $v$ either belong to $P$ or have lists of size four.

Let $y_1$ and $y_2$ be the neighbors of $v$ in $\dom(\vf)$, where $y_1$ is closer to $p_0$ than $y_2$.
Let $Q_1$ and $Q_2$ be the subpaths of $Q$ intersecting in $v$ (where $Q_1$ is closer to $p_0$ than $Q_2$)
and let $Q'_1$ and $Q'_2$ be the paths obtained from them by adding the edge $y_1v$.
For $i\in\{1,2\}$, if $\ell(Q_i)<\ell-1$, then \refclaim{cl-empty} implies that the subgraph of $G$ split off by $Q'_i$
does not contain any vertex of $X$.  Since $X\neq \emptyset$ and $\ell(Q_1)+\ell(Q'_2)=\ell(Q)=\ell$,
it follows that $\ell(Q_1)=1$ or $\ell(Q_2)=1$.

If for some $i\in\{1,2\}$, we have $\ell(Q_i)>1$ and an internal vertex $z$ of $Q_i$ is adjacent to
$y_i$ (this is only possible when $\ell\ge 3$), then $Q_i$ is an edge-disjoint union of paths $Q_{i,1}$ and $Q_{i,2}$ such that
$Q_{i,1}$ together with $vy_iz$ forms a cycle $C$ of length at most $\ell$ and
$Q_{i,2}+zy_i$ is a span of length $k\le \ell-1$.  By considering the interior of $C$ and the subgraph
of $G$ split off by $Q_{i,2}+zy_i$ separately, \refclaim{cl-empty} again implies that the subgraph of $G$ split off by $Q$
does not contain any vertex of $X$.  This is a contradiction.  It follows that no internal vertex of $Q_i$
is adjacent to $y_i$, and thus no internal vertex of $Q_i$ has a list of size four.
Therefore, all internal vertices of $Q$ except for $v$ belong to~$P$.

If $\ell(Q_1)>1$, then let $Q_2=vw$, where $w\in V(H)$; consider the subgraph $F$ of $G$ split off by $y_1vw$.
Note that $\ell=\ell(Q)\ge 3$ and the distance between $v$ and $X$ is at least $D(M,\ell)-\lceil \ell/2\rceil-3\ge D(M,2)$.
By the minimality of $G$, it follows that $F$ (which is $y_1vw$-critical) contains no vertex of $X$.  By Theorem~\ref{thm-prepathw},
we have $\omega_{y_1vw,L}(F)=0$.  This is a contradiction, since in each of the cases (Y2), (Y4a) and (Y4b), $F-\{y_1,v,w\}$
contains a vertex with a list of size four.

Therefore, $\ell(Q_1)=1$.  In case (Y4a), $v$ is not adjacent to $v_1$, and thus $v$ is adjacent to $p_0$.
Similarly, in case (Y4b), $v$ is adjacent to $p_0$, since $v_1$ belongs to $Y$.  Since $v_1$ has a list of size four, it has degree at least
four, and thus at least one vertex of $G$ is drawn inside the $4$-cycle $v_1v_2vp_0$.  This contradicts \refclaim{cl-notr}.
Suppose now that (Y2) holds.  Since $\ell(Q_2)=\ell-1$ and $\deg_{G_{\vf}}(v)\ge3$,
we conclude that $Q_2=vp_2p_3\ldots p_{\ell}$.
Hence, $G_\vf$ is the union of $P$ with the graph $G'=G_\vf-\{p_4,\ldots,p_\ell\}$, and $G'$ is $p_0p_1p_2$-critical with respect to the list assignment $L_\vf$.
Since $v$ is adjacent to $p_0$ and $p_2$, Lemma~\ref{lemma-extthom} implies that $G'$ is equal to the $4$-cycle $p_0p_1p_2v$ with the chord $p_1v$,
and thus $G_{\vf}$ consists of $P$ and the vertex $v$ adjacent to $p_0$, $p_1$ and $p_2$.
By \refclaim{cl-ninp}, $\ell=2$.  Let us postpone the discussion of this case and summarize it in the next claim.
\claim{leftcase1}{In the subcase considered, {\rm (Y2)} holds, $\ell=2$, and $p_0$, $p_1$, $p_2$, $v_1$ and $v_2$ have a common neighbor.}

\item Let us now consider the case that \textbf{two vertices $u, v\in V(G_{\vf})$ with $|L_{\vf}(u)|=|L_{\vf}(v)|=3$ and $|L(v)|>3$ are adjacent.}
By \refclaim{cl-notr}, at most one of $u$ and $v$ has two neighbors in $\dom(\vf)$.  If neither $u$ nor $v$ has two neighbors in $\dom(\vf)$,
then $u,v\in V(H)$ and the choice of $Y$ and $\vf$ ensures that $uv$ is a chord of $H$.  However, that contradicts \refclaim{cl-chord}.
Thus, we can assume that $v$ has two neighbors in $\dom(\vf)$ and $v\not\in V(H)$; and in particular, $Y$ was chosen
according to one of the cases (Y2), (Y4a) or (Y4b) (the case (Y2a) is excluded, since in that case $G_{\vf}$ would contain at most
one vertex with a list of size three).  Since $u$ has at most one neighbor in $\dom(\vf)$ and $|L_{\vf}(u)|=3$, we have $u\in V(H)$.
If $|L(u)|=4$, then $u$ has a neighbor $y\in\dom(\vf)$, and by \refclaim{cl-chord}, we have $uy\in E(H)$.  This is not possible
(in the case (Y4a), the vertex $v_1$ has a list of size four, but it is not adjacent to $v$).  Therefore, $|L(u)|=3$.
Furthermore, inspection of cases (Y2), (Y4a), (Y4b) shows that $u$ has no neighbor in $Y$, as otherwise either $H$ would have a chord contradicting \refclaim{cl-chord},
or $G$ would contain a $4$-cycle $y'yuv$ with $y,y'\in Y$ and $|L(y)|=4$; in the latter case, $y$ would have degree at least four, contradicting \refclaim{cl-notr}.

Let $y_1,y_2\in \dom(\vf)$ be the neighbors of $v$, where $y_1$ is closer to $p_0$ than $y_2$.  By the previous paragraph, $v$ has a neighbor $u'$ in $V(H)\setminus Y$
with $|L(u')|=3$; choose such a neighbor $u'$ so that the subgraph $F$ of $G$ split off by $u'vy_1$ is as small as possible.
Note that $\omega_{u'vy_1,L}(F)\ge 1$, as $|L(y_2)|=4$.  Since $F$ is $u'vy_1$-critical with respect to $L$,
Theorem~\ref{thm-prepathw} implies that $X\cap V(F)\neq\emptyset$, and by the minimality of $G$,
a vertex $x\in X\cap V(F)$ is at distance at most $D(M,2)-1$ from $v$.  Hence, the distance between $x$ and $p$ is at most
$D(M,2)+1+\ell$, and since this is at least $D(M,\ell)$ by the assumptions, we have $\ell\le 2$.

Let $Q$ be the path consisting of $P$, the subpath of $H$ between $p_0$ and $y_1$
and the path $y_1vu'$.  If $|L(v_s)|=|L(v_{s-1})|=3$ and $u'\neq v_s$, include also the edge $p_{\ell}v_s$ in $Q$. Let $G_Q$ be the 
subgraph of $G$ such that $G_Q\cup F=G$ and $G_Q\cap F=u'vy_1$.
Note that $\ell(Q)\le 6$ (even if the edge $p_{\ell}v_s$ has been added to $Q$, since then $\ell=1$ by \refclaim{cl-no33ifg2}).
Since the distance between $v$ and a vertex of $X\cap V(F)$ is at most
$D(M,2)-1$ and $X$ is $M$-scattered, the distance between a central vertex of $Q$ and $X\cap V(G_Q)$
is at least $D(M,6)$.  Since $G_Q$ is $Q$-critical with respect to $L$, the minimality of $G$ implies that $G_Q$ contains no vertex of $X$.

Consider now the graphs $G'_Q=G_Q-Y$ and $F'=F-Y$ with list assignment $L_{\vf}$.
By the choice of $u'$ so that $F$ is minimal, no two adjacent vertices of $F'$ other than $u'$ and $v$ have lists of size three.  Furthermore,
the distance between $v$ and $X$ is at least $M+3>D(M,1)$ by \refclaim{cl-farle2}.  By the minimality of $G$, no $uv$-critical subgraph
of $F'$ (with respect to the list assignment $L_{\vf}$) contains a vertex of $X$, and by Theorem~\ref{thm-thom} we conclude that
every $L_{\vf}$-coloring of $u'v$ extends to an $L_{\vf}$-coloring of $F'$.  Since $G_{\vf}$ is not $L_{\vf}$-colorable, it follows that
$G'_Q$ is not $L_{\vf}$-colorable.  By Theorem~\ref{thm-thom} this is not possible if $\ell=1$, and thus $\ell=2$.

Note that if $xy$ is an edge of $G'_Q$ and $|L_{\vf}(x)|=|L_{\vf}(y)|=3$, then $x$ or $y$ is equal to $v$, and in particular
all such edges distinct from $vu'$ are chords of the outer face of $G'_Q$.
Since $G'_Q$ is not $L_{\vf}$-colorable, it contains a subgraph $G''_Q$ that is $P$-critical with respect to $L_{\vf}$.
By Lemma~\ref{lemma-extthom}, $G''_Q$ is an even fan procession with base $P$ and $L_{\vf}$ is dangerous for $G''_Q$.
Note that all chords of the outer face of $G''_Q$ are incident with $p_1$, and thus $G''_Q$ contains at most one edge $xy$ satisfying $|L_{\vf}(x)|=|L_{\vf}(y)|=3$.
Furthermore, using \refclaim{cl-chord} observe that the outer face of $G''_Q$ has at most one chord, namely $p_1v$.
The only even fan processions satisfying these constraints are a fan of order $1$ or a fat fan of order $2$.
In the former case, $v$ is adjacent to all vertices of $P$.  In the latter case, $v$ is adjacent to $p_0$ and $v_s$,
$|L(v_s)|=3$ and $p_0$, $p_1$, $p_2$, $v$ and $v_s$ have a common neighbor.
This is not possible in the cases (Y4a) and (Y4b), since $v_1$ cannot have degree less than four.
We are left with the case that either the configuration described in \refclaim{leftcase1} appears, or
we have the following:
\claim{leftcase2}{If the situation of \refclaim{leftcase1} does not occur, then {\rm (Y2)} holds, $\ell=2$, the common neighbor $v$ of $v_1$ and $v_2$ is adjacent to $p_0$ and $v_s$, $|L(v_s)|=3$,
and there exists a vertex $w$ adjacent to $V(P)\cup \{v,v_s\}$.}
\end{itemize}

Since either \refclaim{leftcase1} or \refclaim{leftcase2} holds, we always have $\ell=2$ and there exists a vertex $w$ adjacent
to all vertices of $P$, where $w=v$ if  \refclaim{leftcase1} holds. In particular, no two vertices with list
of size three are adjacent in $G$ by \refclaim{cl-no33ifg2}, and $P$ has a unique central vertex.  Therefore, by symmetry we also have $|L(v_s)|=3$, $|L(v_{s-1})|=4$ and $w$ is either adjacent to $v_{s-1}$ and $v_s$, or
adjacent to $v_1$ and the common neighbor $v'$ of $v_{s-1}$, $v_s$ and $p_2$.
Observe that the outcome of \refclaim{leftcase2} contradicts the last conclusion, as $w$ in \refclaim{leftcase2} does not have a neighbor with a list of size four
(thus $w$ is not adjacent to $v_{s-1}$) and $v$ is the only neighbor of $w$ with a list of size five and $v$ is not adjacent
to $p_2$ (excluding the existence of $v'$).

Therefore, \refclaim{leftcase1} holds and $v$ is also adjacent to $v_{s-1}$ and $v_s$.  Let us choose $c_1\in S(v)$ and $c_2\in S(v_1)$ arbitrarily so that $c_1\neq c_2$.
Let $L'$ be the list assignment such that $L'(v_2)=L(v_2)\setminus \{c_2\}$, $L'(v)=\{c_1\}$ and $L'(z)=L(z)$ for any other vertex $z$.
Let $G'=G-\{p_1,p_0,v_1\}$ and $P'=p_2v$.  Note that each $L'$-coloring of $G'$ corresponds to an $L$-coloring of $G$, and thus $G'$ is not $L'$-colorable.
Let $G''$ be a $P'$-critical subgraph of $G'$.  The only possible adjacent vertices of $G''$ with lists of size three are $v_2$ and $v_3$.  Also, the
distance between $v$ and $X$ is at least $D(M,2)-1$.  If $vv_2\in E(G'')$, then $G''$ satisfies the assumptions of Lemma~\ref{lemma-albcorr},
and by the minimality of $G$, we have $X\cap V(G'')=\emptyset$.  However, then $G''$ contradicts Theorem~\ref{thm-thom}.

Finally, suppose that $vv_2\not\in E(G'')$.  Let $d$ be a new color that does not appear in any of the lists, and let $L''$ be the
list assignment obtained from $L'$ by replacing $c_1$ with $d$ in the lists of $v$ and its neighbors in $G''$ and by setting $L''(v_2)=L(v_2)\cup\{d\}$.
Observe that $G''+vv_2$ is not $L''$-colorable, and since no two vertices of $G''+vv_2$ with lists of size three are adjacent, we again
obtain a contradiction with the minimality of $G$ and Theorem~\ref{thm-thom}.
\end{proof}

\section{The conjecture of Albertson}
\label{sec-albertson}

In order to finish the proof of Theorem~\ref{thm-main}, it remains to show that Lemma~\ref{lemma-albimp} holds.
We are going to prove a stronger statement, giving a complete list of the critical graphs where we only forbid
the precolored vertex $x$ to be adjacent with a vertex with a list of size three.  Let us start with a simple observation.

\begin{lemma}\label{lemma-noout}
Let $G$ be a graph drawn in the plane, let $P$ be a path of length at most one contained in the boundary $H$ of the outer face of $G$, and
let $x$ be a vertex of $V(G)\setminus V(P)$.  Let $L$ be a $0$-valid list assignment for $G$ and $X$, where $X=\{x\}$,
and assume that $x$ is not adjacent to any vertex with a list of size three.
If $x\in V(H)$ or $x$ has neighbors only in $H$, then $G$ is $L$-colorable.
\end{lemma}
\begin{proof}
Let $L'$ be the list assignment obtained from $L$ by removing $L(x)$ from the list of all its neighbors
and let $G'=G-x$.  Since $x$ is not adjacent to any vertex with a list of size three, we have $|L'(v)|\ge 3$ for each $v\in V(G)\setminus V(P)$.
Since $L$ is a $0$-valid assignment, $P$ is $L'$-colorable.  Furthermore, by the assumptions of the lemma, all the vertices with lists of size less than five
are incident with the outer face of $G'$.  Hence, $G'$ is $L'$-colorable by Theorem~\ref{thm-thom}, and this coloring extends to an $L$-coloring of $G$.
\end{proof}

We use a part of a result of Dvo\v{r}\'ak et al.~\cite{5choosfar} regarding the case that a path
of length three is precolored, but adjacent vertices with lists of size three are forbidden
(let us remark that this result can also be easily derived from Theorem~\ref{thm-prepathw} and the technique
used in the proof of Lemma~\ref{lemma-boundsize}).
An \emph{obstruction} is a plane graph with a prescribed subpath of its outer face boundary and
prescribed lengths of lists.  An obstruction $O$ \emph{appears} in a graph $G$ with the list
assignment $L$ and a specified path $P$ if $O$ is isomorphic to a subgraph of $G$
such that the prescribed subpath of $O$ corresponds to $P$ and the sizes of the lists given
by $L$ match those prescribed by $O$.  In figures, the full-circle vertices belong to the prescribed path (and their list sizes are not prescribed),
triangle vertices have lists of size three, square vertices have lists of size four and pentagonal vertices have lists of size five.

\begin{figure}
\begin{center}
\includegraphics[width=120mm]{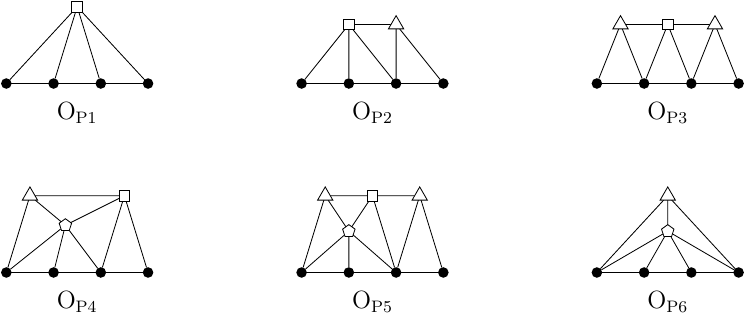}
\end{center}
\caption{Forbidden configurations of Theorem~\ref{thm-5choosfar}.}
\label{fig-obst}
\end{figure}

\begin{theorem}[Dvo\v{r}\'ak et al.~\cite{5choosfar}, Theorem 7]\label{thm-5choosfar}
Let $G$ be a graph drawn in the plane, let $P$ be a path of length at most three contained in the boundary $H$ of the outer face of $G$
and let $L$ be a valid list assignment for $G$ such that no two vertices with lists of size three are adjacent
and all vertices of $P$ have lists of size one.
If the following conditions hold, then $G$ is $L$-colorable:
\begin{itemize}
\item if a vertex $v$ has three neighbors $w_1,w_2,w_3$ in $V(P)$, then $L(v)\neq L(w_1)\cup L(w_2)\cup L(w_3)$, and
\item if $O$ is an obstruction depicted in Figure~\ref{fig-obst} that appears in $G$, then $O$ is $L$-colorable.
\end{itemize}
\end{theorem}

Another special case of the result of Dvo\v{r}\'ak et al.~\cite{5choosfar} is the following:

\begin{theorem}[Dvo\v{r}\'ak et al.~\cite{5choosfar}, Theorem 7]\label{thm-onefour}
Let $G$ be a graph drawn in the plane, let $P$ be a path of length at most one contained in the boundary $H$ of the outer face of $G$,
let $w$ be a vertex in $V(G)\setminus V(H)$ and let $L$ be a list assignment for $G$ such that
$P$ is $L$-colorable,
all vertices in $V(G)\setminus V(H)$ other than $w$ have lists of size at least five,
all vertices in $V(H)\setminus V(P)$ have lists of size at least three and
no two vertices with lists of size three are adjacent.
If\/ $|L(w)|=4$, then $G$ is $L$-colorable.
\end{theorem}

We also use the following characterization of critical graphs with a short precolored face, given in \cite{bohmelc} (this can also
be easily derived from Lemma~\ref{lemma-preouf}).

\begin{lemma}[\cite{bohmelc}]\label{lemma-critface}
Let $G$ be a plane graph, let $H$ be the boundary of its outer face, and let $L$ be a list assignment such that $|L(v)|\ge 5$ for $v\in V(G)\setminus V(H)$.
If $H$ is an induced cycle of length at most six and $G$ is $H$-critical with respect to $L$,
then $|H|\ge 5$ and one of the following holds:
\begin{itemize}
\item $|V(G)\setminus V(H)|=1$, or
\item $|H|=6$ and $V(G)\setminus V(H)$ consists of two adjacent vertices of degree five, or
\item $|H|=6$ and $V(G)\setminus V(H)$ consists of three pairwise adjacent vertices of degree five.
\end{itemize}
\end{lemma}

Let us now proceed with a strengthening of Lemma~\ref{lemma-albimp}.

\begin{figure}
\begin{center}
\includegraphics[width=120mm]{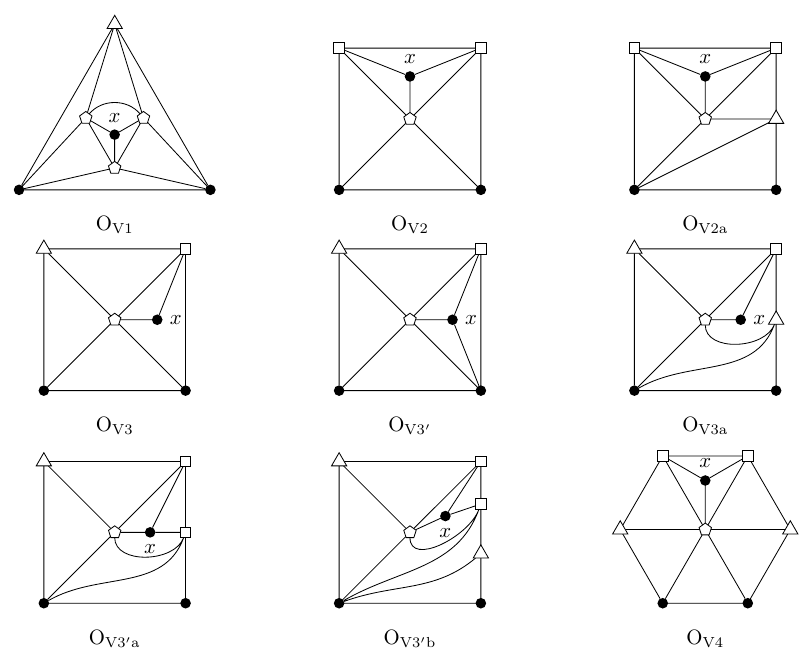}
\end{center}
\caption{Forbidden configurations of Lemma~\ref{lemma-albimps}.}
\label{fig-albimpob}
\end{figure}

\begin{lemma}\label{lemma-albimps}
Let $G$ be a graph drawn in the plane, let $P$ be a path of length at most one contained in the boundary $H$ of the outer face of\/ $G$ and
let $x$ be a vertex of\/ $V(G)\setminus V(P)$.  Let $L$ be a $0$-valid list assignment for $G$ and $X$, where $X=\{x\}$.
If the following conditions hold, then $G$ is $L$-colorable:
\begin{itemize}
\item no two vertices with lists of size three are adjacent,
\item $x$ is not adjacent to any vertex with a list of size three, and
\item if $O$ is an obstruction drawn in Figure~\ref{fig-albimpob} that appears in $G$ (with the marked vertex corresponding to $x$),
then $O$ is $L$-colorable.
\end{itemize}
\end{lemma}

\begin{proof}
We can assume that $|L(p)|=1$ for $p\in V(P)$ and that $\ell(P)=1$.  Let $P=p_0p_1$.
Observe that in the process of reducing lists of vertices in $P$ (in order to be able to assume that $|L(p)|=1$), we may create a non-colorable obstruction from Figure~\ref{fig-albimpob}. However, $G$ contains at most one such obstruction $O$ (with the exception of \ob{V3} and \ob{V3'}, when we consider $O=\hbox{\ob{V3'}}$). Therefore, we can always take the lists for vertices in $V(P)$ coming from an $L$-coloring of $O$.

For contradiction, assume that $G$ is a counterexample
with $|V(G)|+|E(G)|$ the smallest possible.  Note that $G$ is $P$-critical and connected.
By Lemma~\ref{lemma-noout}, we have $x\not\in V(H)$.

By Lemma~\ref{lemma-crs} and Theorem~\ref{thm-5choosfar}, every non-facial triangle in $G$ and every chord of $H$ separates $P$ from $x$.
Furthermore, by applying these results to $2$-chords, we obtain the following.

\claim{cl-v-2chord}{If a $2$-chord $Q=v_1v_2v_3$ of $H$ does not separate $P$ from $x$, then the subgraph of $G$ split off by $Q$
consists either of the edge $v_1v_3$, or of a single vertex with a list of size three adjacent to $v_1v_2v_3$.}

Suppose that $G$ contains a vertex cut of size one, and let $G_1$ and $G_2$ be the subgraphs of $G$ such that $G=G_1\cup G_2$,
$G_1$ and $G_2$ intersect in a single vertex $v$ and $P\subseteq G_1$.  Note that $x\in V(G_2)\setminus \{v\}$, and by
the minimality of $G$ and Lemma~\ref{lemma-crs}, we conclude that $G_2$ consists of the edge joining $v$ with $x$.
By Lemma~\ref{lemma-noout}, we have $v\not\in V(H)$.  It follows that $H$ is a cycle.

Let us now consider a chord $uv$ of $H$ and let $G_1$ and $G_2$ be the subgraphs of $G$ intersecting in $uv$,
where $P\subseteq G_1$.  By Lemma~\ref{lemma-crs} and the minimality of $G$, either $G_2$ is one of the graphs drawn in Figure~\ref{fig-albimpob},
or $V(G_2)=\{u,v,x\}$.  The latter is impossible by Lemma~\ref{lemma-noout}; hence, assume the former.
The inspection of these graphs shows that there exists only one proper $L$-coloring $\vf$ of the subgraph of $G$ induced by $\{u,v,x\}$ that does not extend to an $L$-coloring
of $G_2$ (let us recall that $|L(x)|=1$).  By symmetry, we can assume that $x$ is not adjacent to $u$ and that $u$ does not have a list of size three.
Let $G'$ be the graph obtained from $G_1$ by adding a new vertex $x'$ and the edge $ux'$, and if $vx\in E(G)$, then also
the edge $vx'$.  Let $c$ be a new color that does not appear in any of the lists and let $L'$ be
a list assignment for $G'$ defined as follows: $L'(x')=\{c\}$, $L'(u)=(L(u)\setminus \{\vf(u)\})\cup \{c\}$,
if $vx\in E(G)$, then $L'(v)=(L(v)\setminus L(x))\cup \{c\}$, otherwise $L'(v)=L(v)$, and $L'(w)=L(w)$ for
every $w\in V(G_1)\setminus\{u,v\}$.  Since each $L'$-coloring of $G'$ corresponds to an $L$-coloring of $G$,
it follows that $G'$ is not $L'$-colorable.  By the minimality of $G$, this is only possible if $u\in V(P)$ and $L(u)=\{\vf(u)\}$.
By Theorem~\ref{thm-5choosfar} applied to $G_1$ with respect to the path $P+uv$, we conclude that $G_1$ either is a triangle,
or $V(G_1)=\{p_0,p_1,v,w\}$ for some vertex $w$ with a list of size three adjacent to $p_0$, $p_1$ and $v$.
However, it is easy to check that the composition of $G_1$ with $G_2$ (an obstruction from Figure~\ref{fig-albimpob}) is either $L$-colorable or equal to one
of the obstructions in Figure~\ref{fig-albimpob}.  This is a contradiction.  Consequently:

\claim{cl-v-chords}{$H$ is an induced cycle.}

Suppose that the distance between $x$ and $P$ is $1$, say $xp_0\in E(G)$.  Observe that $xp_1\not\in E(G)$ by
Lemma~\ref{lemma-crs} and Theorem~\ref{thm-5choosfar}.  Let $G'$ be the graph obtained from $G$ by splitting $p_0$
to two vertices $p_0'$ and $p_0''$, where both $p_0'$ and $p_0''$ are adjacent to $x$, $P'=p_0'xp_0''p_1$ is a path in $G'$
and other neighbors of $p_0$ are divided between $p_0'$ and $p_0''$ in the planar way.
Note that $G'$ is $P'$-critical and we can apply Theorem~\ref{thm-5choosfar} for it.  Using \refclaim{cl-v-chords}, observe that $P'$ is an induced path and that
each vertex with a list of size three has at most two neighbors in $P'$, hence $G'$ is one of the graphs drawn in
Figure~\ref{fig-obst}.  Since each vertex distinct from $x$ is adjacent to at most one of $p_0'$ and $p_0''$ and $x$ is not adjacent to a vertex with a list of size three, it
follows that $G'$ is \ob{P4}.  But then $G$ is \ob{V3'}.
Therefore, we have:

\claim{cl-v-dist}{The distance between $x$ and $P$ is at least\/ $2$.}

\begin{figure}
\begin{center}
\includegraphics[width=30mm]{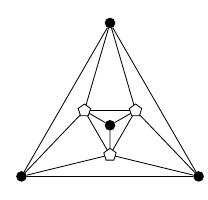}
\end{center}
\caption{Nontrivial separating triangle.}
\label{fig-septri}
\end{figure}

Consider a separating triangle $C=v_0v_1v_2$ in $G$, and let $G'=\cin_C(G)$ with the list assignment $L$.  Note that $G'$ is $C$-critical.
By Theorem \ref{thm-thom} it follows that $x\in V(G')\setminus V(C)$.
If $x$ is adjacent to say $v_0$, then $G'$ is bounded by the closed walk $xv_0v_1v_2v_0x$ of length $5$.
Formally, we split $v_0$ into two vertices $v_0'$ and $v_0''$ as we did with $p_0$ in the previous paragraph.
Observe that $|V(G')\setminus \{x,v_0,v_1,v_2\}|\neq 1$, since no vertex other than $x$ is adjacent to both $v_0'$ and $v_0''$
and all vertices in $V(G')\setminus \{x,v_0,v_1,v_2\}$ have degree at least five.  
Using Lemma~\ref{lemma-critface}, we conclude that $V(G')=\{v_0,v_1,v_2,x\}$.

Let us now consider the case that the distance between $C$ and $x$ is at least two.
Let $\vf$ be an $L$-coloring of $C$ that is obtained from an $L$-coloring of $G-x$. Then $\vf$ does not extend to an $L$-coloring of $G'$.
Let $L'$ be the list assignment such that $L'(v_0)=\{\vf(v_0)\}$, $L'(v_1)=\{\vf(v_1)\}$, $L'(v_2)=\{\vf(v_0),\vf(v_1),\vf(v_2)\}$
and $L'$ matches $L$ on the remaining vertices of $G'$.  Then $G'$ is not $L'$-colorable. By the minimality of $G$, \ob{V1}
appears in $G'$.  We conclude that $G'$ is the graph drawn in Figure~\ref{fig-septri}.  In that case, every $L$-coloring of $C$ different
from $\vf$ extends to an $L$-coloring of $G'$.  Since $G$ is not \ob{V1}, at least one vertex $w\in V(C)\setminus V(P)$ does not
have a list of size three.  Let $G''$ be the graph obtained from $G-(V(G')\setminus V(C))$ by adding a vertex $x'$ and the edge $x'w$.
Let $L''$ be the list assignment such that $L''(x')=\{\vf(w)\}$ and $L''$ matches $L$ on other vertices of $G''$.  Note that $G''$
is not $L''$-colorable, since each such coloring would extend to an $L$-coloring of $G$.  Furthermore, $x'$ has degree one,
hence $G''$ does not contain any of the obstructions.  Therefore, $G''$ contradicts the minimality of $G$.
This implies:

\claim{cl-v-septri}{If\/ $C$ is a separating triangle in $G$, then $V(\cin_C(G))=V(C)\cup\{x\}$.}

Let $p_1p_0v_1\ldots v_s$ be the facial walk of $H$.  Note that $s\ge 1$.
For $v\in V(G)\setminus (V(P)\cup \{x\})$, let 
$$S(v)=L(v)\setminus \bigcup_{\substack{u\in V(P)\cup \{x\}\\ uv\in E(G)}} L(u).$$
Observe that $|S(v)|=|L(v)|-k$, where $k$ is the number of neighbors of $v$ in $V(P)\cup \{x\}$, by the minimality of $G$.

If $s=1$, then note that $|S(v_1)|\ge 1$. Choose an arbitrary color $c\in S(v_1)$ and let $L'$ be the list assignment obtained from $L$ by
removing $c$ from the lists of neighbors of $v_1$.  Note that $G-v_1$ is not $L'$-colorable, and since it contains no vertices with list
of size three, by the minimality of $G$ we conclude that \ob{V2} appears in $G-v_1$.  But then $G$ is equal to \ob{V1}, which is a contradiction.
Therefore, $s\ge 2$.  As we observed before, $x\not\in V(H)$, and hence neither $v_1$ nor $v_2$ is equal to $x$.

Suppose that $|L(v_1)|=5$, or that $|L(v_1)|=4$ and $|L(v_2)|\ge 4$.  If $x$ is not adjacent to $v_1$, then let $y=p_0$, otherwise let $y=x$.
Let $L'$ be the list assignment obtained from $L$ by removing $L(y)$ from $L(v_1)$ and let $G'=G-yv_1$.
Observe that $G'$ is $P$-critical with respect to $L'$, and by the minimality of $G$, it is one of the graphs in Figure~\ref{fig-albimpob}.
This is not possible if $y=p_0$, since then $G'$ either is not $2$-connected or contains a vertex (a neighbor of $p_1$) with a list of size $5$ incident with the outer face.
If $y=x$, then $v_1$ and $x$ are not adjacent in $G'$, but the edge $v_1x$ can be added keeping the graph planar. This is only possible if $G'$ is either \ob{V3} or \ob{V3a}.
However, in such a case $G$ would be isomorphic to \ob{V2}, \ob{V2a} or \ob{V3'a}, and would be $L$-colorable by assumption.
We conclude the following.

\claim{cl-v-bsizes}{We have $s\ge2$, $|L(v_1)|\le 4$, and one of $v_1$ and $v_2$ has a list of size three; symmetrically, $|L(v_s)|\le 4$ and one of $v_s$ and $v_{s-1}$ has a list of size three.}

Next, we claim the following:

\claim{cl-v-common}{If the vertices $p_0$, $p_1$, $v_1$ and $v_2$ have a common neighbor $w$, then $w$ is adjacent to $x$.}

\begin{proof}
By \refclaim{cl-v-chords} and \refclaim{cl-v-dist}, we have $w\in V(G)\setminus (V(H)\cup\{x\})$.  Suppose that $w$ is not
adjacent to $x$.  By Lemma~\ref{lemma-noout} and \refclaim{cl-v-septri}, we conclude that $wp_1p_0$, $wp_0v_1$ and $wv_1v_2$
bound faces.  Let $G_1$ and $G_2$ be the subgraphs of $G$ intersecting in $p_1wv_2$ such that $P\subset G_1$.  Note that $x\in V(G_2)$,
$|L(v_1)|=3$, $|S(v_1)|=2$ and $|S(w)|=3$.  Choose $c\in S(w)\setminus S(v_1)$ arbitrarily.
Let $L'(w)=\{c\}$ and $L'(v)=L(v)$ for any $v\in V(G_2)\setminus \{w\}$.  Note that every $L'$-coloring of $G_2$ extends to an $L$-coloring of $G$, hence $G_2$ is not $L'$-colorable.  By the minimality of $G$ and \refclaim{cl-v-dist},
we conclude that one of the obstructions $K$ drawn in Figure~\ref{fig-albimpob} appears in $G_2$, with the precolored path
$p_1w$.  Note that $v_2\in V(K)$, as otherwise $G$ contains a $2$-chord contradicting \refclaim{cl-v-2chord}. Also, $|L'(v_2)|\ge 4$, as
$v_2$ is adjacent to $v_1$, which has a list of size three.  Therefore,
$K$ is one of \ob{V2}, \ob{V2a}, \ob{V3} or \ob{V3'a} (the case \ob{V3'} is excluded, since $x$ is adjacent neither to $w$ nor to $p_1$).
By \refclaim{cl-v-bsizes}, $K$ is not \ob{V2}.  Furthermore, $H$ is not
\ob{V3'a}, since $w$ has degree at least five.  In the remaining two cases, \refclaim{cl-v-chords}, \refclaim{cl-v-septri} and the assumption that $x$ is not
adjacent to a vertex with a list of size 3 imply that $K=G_2$. It is a simple exercise to check that the combination of $G_1$ with $K$ is $L$-colorable.  This is a contradiction.
\end{proof}

Suppose that a vertex of $P$, say $p_0$, has degree two in $G$.  We can assume that the color of $p_1$ only appears in the lists of its
neighbors, by replacing it with a new color if necessary.  Let $G'=G-p_0+p_1v_1$ and let $L'$ be the list assignment such that
$L'(v_1)=(L(v_1)\setminus L(p_0))\cup L(p_1)$ and $L'$ matches $L$ on other vertices.  By the minimality of $G$, we have that
$G'$ is $L'$-colorable; but this gives an $L$-coloring of $G$, which is a contradiction.
Therefore, both vertices of $P$ have degree at least three.

\claim{cl-v-inside}{The vertex $x$ has no neighbor in $H$.}

\begin{proof}
Suppose that $x$ has a neighbor $v\in V(H)$.  By \refclaim{cl-v-dist}, we have $v\not\in V(P)$.

Let us first consider the case that $p_0$, $p_1$ and $x$ have a common neighbor $w$.  For $i\in \{0,1\}$,
let $Q_i=p_iwxv$ and let $G_i$ be the subgraph of $G$ split off by $Q_i$.
Note that $G_i$ is $Q_i$-critical, and we can apply Theorem~\ref{thm-5choosfar} to it.
Suppose that each of $G_0$ and $G_1$ is among the graphs drawn in Figure~\ref{fig-obst} different from \ob{P6},
or consists of a vertex with a list of size three adjacent to $p_i$, $w$ and $v$,
or consists of an edge joining $v$ to $p_i$.
(Note that some configurations of Figure~\ref{fig-obst} are excluded since $x$ is not adjacent to a vertex with a list of size~3.) 
A straightforward case analysis shows that for any $c\in S(w)$, there exists at most one color $c'\in S(v)$
such that the $L$-coloring of $Q_i$ that assigns the color $c$ to $w$ and the color $c'$ to $v$ does not
extend to an $L$-coloring of $G_i$.  Since $|L(v)|\ge 4$, we conclude that $G$ is $L$-colorable, which is a contradiction.
Hence, we can assume that say $G_0$ does not satisfy this property; by Theorem~\ref{thm-5choosfar}, the following cases
are possible:
\begin{itemize}
\item $G_0$ contains the edge $vw$
and either the edge $p_0v$ or a vertex with a list of size three adjacent to $p_0$, $w$ and $v$; or,
\item $G_0$ is \ob{P6}.
\end{itemize}
In the former case, $vw\not\in E(G_1)$, since $G$ does not have parallel edges.  In the latter case,
if $vw\in E(G_1)$, then $G$ is easily seen to be $L$ colorable.  Therefore, we can assume that $vw\not\in E(G_1)$.
If $G_1$ is \ob{P6}, then all the combinations with the possible choices for $G_0$ result in an $L$-colorable graph.
Hence, for any $c\in S(w)$ there exists at most one color $c'\in S(v)$ such that the corresponding coloring does
not extend to an $L$-coloring of $G_1$.  If $G_0$ is \ob{P6}, this would imply that $G$ is $L$-colorable.
Therefore, $G_0$ contains the edge $vw$.  A straightforward case analysis shows that all the remaining combinations
of the choices for $G_0$ and $G_1$ result in \ob{V2}, \ob{V3}, \ob{V4} or in an $L$-colorable graph.
We conclude that $p_0$, $p_1$ and $x$ do not have a common neighbor.

Let $M$ be the set consisting of $V(P)$ and of all vertices with lists of size three adjacent to $P$, i.e., $M\subseteq\{v_s,p_1,p_0,v_1\}$.
Suppose that a vertex $w$ has at least three neighbors in $M$; note that $w\not\in V(H)$ by \refclaim{cl-v-chords}.
If $w$ is adjacent to $v_s$ and $p_0$, then since $x$ has a neighbor in $V(H)$, we have $x\not\in V(\cin_{p_0p_1v_sw}(G))$,
and since $p_1$ has degree at least three, Lemma~\ref{lemma-critface} implies that $w$ is also adjacent to $p_1$.
Hence, by symmetry, we can assume that $w$ is adjacent to $p_0$, $p_1$ and $v_1$.
By the previous paragraph, $w$ is not adjacent to $x$, and thus $|S(w)|=3>|S(v_1)|$.
Let $c$ be a color in $S(w)\setminus S(v_1)$.  Let $G'=G-p_0$ and let $L'$ be the list
assignment such that $L'(w)=\{c\}$, $L'(v_1)=S(v_1)\cup\{c\}$ and $L'$ matches $L$ on all other vertices.
Then $G'$ is not $L'$-colorable, and thus one of the configurations $K$ drawn in Figure~\ref{fig-albimpob} appears in $G'$.
Observe that \refclaim{cl-v-2chord}, \refclaim{cl-v-chords}, \refclaim{cl-v-septri} and \refclaim{cl-v-common} imply that $G'=K$.  Since $w$ has degree at least five
and $|L'(v_1)|=3$,
this is only possible if $K$ is \ob{V1}, \ob{V3a}, \ob{V3'a} or \ob{V3'b}.  However, the corresponding graph $G$ is easily seen to be $L$-colorable,
which is a contradiction.

Consequently, every vertex has at most two neighbors in $M$.  Furthermore, each vertex in $V(H)$ other than $v$ has at most
one neighbor in $M$, by \refclaim{cl-v-chords}.  Let $\theta$ be an $L$-coloring of the subgraph of $G$ induced by $M$.
If $v$ is the only neighbor of $x$ in $V(H)$, then
let $G'$ be the graph obtained from $G-M$ by splitting $v$ into two vertices $v'$ and $v''$
adjacent to $x$, with other neighbors of $v$ distributed between $v'$ and $v''$ in the planar way.  Let $L'$ be the list assignment
such that $L'(v')=L'(v'')$ consists of a single color in $S(v)$ distinct from the colors of the neighbors of $v$ in $M$ according
to $\theta$ and $L'(z)=L(z)\setminus\{\theta(t): t\in M, tz\in E(G)\}$ for any other vertex $z\in V(G')$.
If $x$ has at least two neighbors in $V(H)$, then by \refclaim{cl-v-2chord}, it has exactly two such neighbors $v'$ and $v''$
and $v'v''\in E(G)$.  In this case, let $G'=G-M-v'v''$, and let $L'$ be defined as before for vertices other than
$v'$ and $v''$, with the lists of $v'$ and $v''$ chosen to consist of a single color distinct from each other and the colors of their
neighbors in $M\cup\{x\}$.  Let $P'=v'xv''$.  Note that each vertex of $G'$ not in $P'$ has a list of size at least three, and all internal vertices
have lists of size five.  Since any $L'$-coloring of $G'$ together with $\theta$ would give an $L$-coloring of $G$,
no such $L'$-coloring exists. By Lemma~\ref{lemma-extthom}, $G'$ contains an even fan procession $F$ with base $P'$ and $L'$ is a dangerous
assignment for it.

Suppose first that $F$ is not a fan.  By Observation~\ref{obs-procprop}, there is at most one coloring of $P'$ that does not extend to an $L'$-coloring of $F$.
If there were at least two choices for the colors of the endvertices of $P'$, at least one of them would give a coloring of $G'$
extending to an $L$-coloring of $G$.  Therefore, there is only one choice, which is only possible if $v$ is the only neighbor of $x$ in
$H$ and $v$ has two neighbors in $M$.  In this case, $v'$ and $v''$ have the same color, hence $F$ is not a fat fan and $x$ has
a neighbor $z$ with $|L'(z)|=3$.  Since all neighbors of $x$ in $V(H)$ belong to $P'$,
we conclude that $z$ has two neighbors in $M$.  Let $C$ be the cycle formed by $z$, its neighbors in $M$ and the path between them in $M$;
by choosing the neighbors of $z$ as close to each other as possible, we can assume that $C$ is an induced cycle.
Note that $|C|\le 5$.  Since no vertex has more than two neighbors in $M$, Lemma~\ref{lemma-critface} applied to $\cin_C(G)$ with the list assignment $L$
implies that $C$ bounds a face of $G$.  Because vertices of $P$ have degree at least three and $x$, $p_0$ and $p_1$ do not have a common neighbor,
we conclude that $z$ is adjacent either to $p_0$ and $v_1$, or to $p_1$ and $v_s$.  By symmetry, we can assume the former.  Since there is only one
choice for the color of $v'$ and $v''$, we have $v=v_2$
and $V(H)=M\cup \{v_2\}$.  Since $F$ is not a fan, it has at least three vertices with lists of size three.  By planarity, we conclude that it has exactly three,
namely the vertex $z$, a vertex adjacent to $p_0$ and $p_1$, and
a vertex adjacent to $p_1$ and $v_s$, where $s=3$.  Consequently, $F$ consist of the triangle $xv'z$ and of a fat fan of order two.
See Figure~\ref{fig-fanandx}(a).
However, the corresponding graph $G$ is $L$-colorable.

\begin{figure}
\begin{center}
\includegraphics[width=120mm]{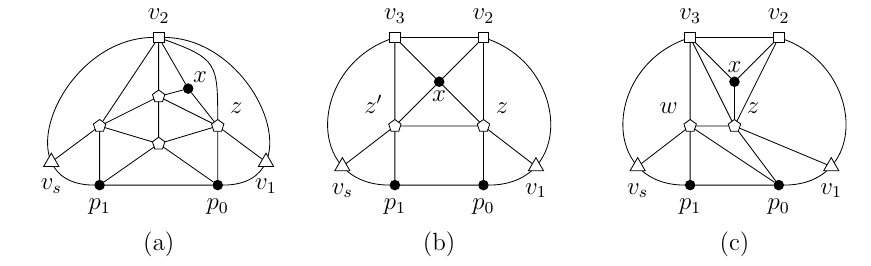}
\end{center}
\caption{Configurations from \refclaim{cl-v-inside}.}
\label{fig-fanandx}
\end{figure}

Therefore, $F$ is a fan.  Let $z$ be a neighbor of $x$ with $|L'(z)|=3$ that is also adjacent to the endvertex $v'$ of $P'$.
Again, $z$ has two neighbors in $M$ and we can assume that they are $p_0$ and $v_1$.
Let us now consider the case that $F$ has order at least two. Then, there exists a common neighbor $z'$ of $x$ and $v''$ distinct from $z$,
and $z'$ is adjacent to $p_1$ and $v_s$.  By \refclaim{cl-v-2chord}, we have that $v'$ is adjacent to $v_1$ and $v''$ is adjacent
to $v_s$.
We apply Lemma~\ref{lemma-critface} to the subgraph $\cin_{p_0zxz'p_1}(G)$ with list assignment $L$.
Since $p_0$, $p_1$ and $x$ do not have a common neighbor, we conclude that
the $5$-cycle $p_0zxz'p_1$ is not induced, and furthermore, that $p_0z',p_1z\not\in E(G)$.
Since both $z$ and $z'$ have degree at least five, it follows that $zz'\in E(G)$.
There are two cases depending on whether $s=3$ or $s=4$, but in both of them, the resulting graph $G$ is $L$-colorable. See Figure~\ref{fig-fanandx}(b) illustrating the case $s=4$.

We conclude that $F$ is a fan of order one; hence, $z$ is adjacent to both ends of $P'$.  It follows that $x$ has two neighbors ($v_2$ and $v_3$) in $H$.
If we have three possible choices for the list (color) of $v''=v_3$ in $L'$, then we can choose the list so that $F$ is $L'$-colorable and obtain an $L$-coloring of $G$.
This is a contradiction, hence $v_3$ has a neighbor in $M$.  Note that $s>3$; otherwise, since $p_1$ has degree at least three, the $4$-cycle
$p_1p_0zv_3$ would by Lemma~\ref{lemma-critface} applied to $\cin_{p_1p_0zv_3}(G)$ have a chord $p_1z$, contradicting the observation that $z$ has only two neighbors in $M$.  Therefore, we have $s=4$.
Similarly, since $p_1$ has degree at least three and it is not adjacent to $z$ or $v_3$, Lemma~\ref{lemma-critface}
applied to the interior of the $5$-cycle $p_1p_0zv_3v_4$ implies that 
there exists a vertex $w$ adjacent to $p_0$, $p_1$, $z$, $v_3$ and $v_4$.  
See Figure~\ref{fig-fanandx}(c).  However, $w$ has three neighbors in $M$, which we already excluded.
This completes the proof of \refclaim{cl-v-inside}.
\end{proof}

Let the set $Y$, its partial coloring $\vf$, and the list assignment $L_{\vf}$ be defined according to Definition~\ref{def-y}.
Note that $G-Y$ is not $L_{\vf}$-colorable, and thus $G-Y$ contains a $P$-critical subgraph $G_\vf$ (with respect to the
list assignment $L_\vf$).
Let us note that \refclaim{cl-v-chords} together with the choice of $Y$ implies that every vertex
$v\in V(G_{\vf})\setminus \{p_0,p_1,x\}$ satisfies $|L_\vf(v)|\ge 3$.

\claim{cl-v-main}{There exists a neighbor of $x$ adjacent to two vertices of $\dom(\vf)$.}

\begin{proof}
Note that $x$ is not adjacent to $Y$ or any other vertex of $H$ (\refclaim{cl-v-inside}). By excluding the conclusion of the claim,
$x$ is not adjacent to a vertex with a list of size three in $G_{\vf}$.  By Theorem~\ref{thm-thom}, we have $x\in V(G_{\vf})$.

\begin{itemize}
\item Let us first consider the situation
that $G_{\vf}$ contains two adjacent vertices $u$ and $v$ with $|L_\vf(u)|=|L_\vf(v)|=3$.  If $u,v\in V(H)$, then \refclaim{cl-v-chords}
implies that $uv\in E(H)$ and the choice of $Y$ would ensure that either $|L_\vf(u)|>3$ or $|L_\vf(v)|>3$.  Hence, by symmetry we can
assume that $u\not\in V(H)$, and thus $u$ has two neighbors in $\dom(\vf)$, and one of (Y2), (Y4a) and (Y4b)
occurs.  If $v\not\in V(H)$, then $v$ is also adjacent to the two vertices $y_1,y_2\in \dom(\vf)$.  By \refclaim{cl-v-septri}, we have $y_1y_2\not\in E(G)$,
hence (Y4b) happens, $y_1=v_1$ and $y_2=v_3$.  Choose vertices $u$ and $v$ so that the disk bounded by the $4$-cycle $C=v_1v_2v_3u$ is as small as possible.
By Theorem~\ref{thm-5choosfar}, we have $x\in V(\cin_C(G))\setminus V(C)$.  However, $u$ is not adjacent to $x$ and forms a vertex cut in $G_{\vf}$.
Let $G_1$ and $G_2$ be the subgraphs of $G_{\vf}$ such that $G_1\cap G_2=u$, $G_1\cup G_2=G_\vf$, both $G_1$ and $G_2$ have
at least two vertices, $P\subset G_1$ and $x\in V(G_2)$.
By Theorem~\ref{thm-thom}, $G_1$ is $L_{\vf}$-colorable, and by the minimality of $G$, the precoloring of $u$ given by this
coloring extends to an $L_{\vf}$-coloring of $G_2$ (the choice of $u$ ensures that no two vertices with lists of size three
are adjacent in $G_2$).  Hence, $G_{\vf}$ is $L_{\vf}$-colorable.  This is a contradiction.

We conclude that there exists a vertex $u\in V(G)\setminus V(H)$
adjacent to two vertices in $\dom(\vf)$ and all edges joining vertices with lists of size three (according to $L_\vf$) are incident with $u$.
Furthermore, the other ends of these edges belong to $H$.  Choose a neighbor $v'$ of $u$ incident with the outer face of $G_{\vf}$ so that
$|L_\vf(v')|=3$ and the subgraph $G_2$ of $G_\vf$ split off by $uv'$ is as large as possible.
Note that all edges joining $u$ to vertices with lists of size three belong to $G_2$.
Let $G_1$ be the subgraph of $G_\vf$ such that $G_1\cup G_2=G_\vf$ and $G_1\cap G_2=uv'$
(we have $P\subseteq G_1$). Let $P_2=uv'$.  Note that $G_2\neq uv'$, since otherwise $G$ would contain a $2$-chord $Q$ consisting
of $uv'$ and a vertex in $\dom(\vf)$ contradicting \refclaim{cl-v-2chord}.  By Theorem~\ref{thm-thom},
we have that $x\in V(G_2)$.

Note that $G_2$ is $P_2$-critical with respect to $L_{\vf}$ and $x$ is not adjacent to any vertex of $P_2$.  Furthermore, no two vertices in $V(G_2)\setminus V(P_2)$
with lists of size three are adjacent.  Observe that $x$ is adjacent neither to $u$ nor to $v'$, thus by the minimality of $G$, we conclude that $G_2$
is equal to one of the graphs drawn in Figure~\ref{fig-albimpob}.  In particular, there exists a unique coloring $\psi$ of $P_2$ that does not extend to an $L_{\vf}$-coloring of
$G_2$.

Suppose that $u$ and $v'$ do not have a common neighbor $w$ in $G_1$ with $|L_{\vf}(w)|=4$. 
Note that by \refclaim{cl-v-2chord} and \refclaim{cl-v-common}, $u$ is not adjacent to both vertices of $P$.
Let $c_u\in L_{\vf}(u)$ be a color that is 
different from $\psi(u)$ and different from the color of a neighbor of $u$ in $P$, if it has any.  Let $c_{v'}\in L_{\vf}(v')$ be a color different from $c_u$
and different from the color of a neighbor of $v'$ in $P$, if it has any.  Let $G'_1=G_1-\{u,v'\}$ with the list $L_1$ obtained from $L_{\vf}$ by removing $c_u$ from the
lists of neighbors of $u$ and $c_{v'}$ from the lists of neighbors of $v'$.  By the choice of $v'$ and the assumption that $u$ and $v'$ have no common neighbor with a list of size 4, 
it follows that every vertex of $V(G'_1)\setminus V(P)$ has a list of size at least three.  By Theorem~\ref{thm-thom}, we conclude that
$G_1$ has an $L$-coloring such that the color of $u$ is not $\psi(u)$.  However, this coloring extends to an $L_{\vf}$-coloring of $G_{\vf}$,
which is a contradiction.

Therefore, we can assume that $u$ and $v'$ have a common neighbor $w\in V(G_1)$ with $|L_{\vf}(w)|=4$.  
If $w\in V(H)$, then \refclaim{cl-v-chords} and \refclaim{cl-v-septri} imply that $v'$ has degree two in $G_1$.
On the other hand, if $w\not\in V(H)$, then $w$ has a neighbor in $\dom(\vf)$ and \refclaim{cl-v-septri} implies
that $u$ has degree two in $G_1$.  In the former case, let $q_0=v'$ and $q_1=u$, in the latter case let $q_0=u$ and $q_1=v'$.
Note that $q_0$ has no neighbor in $P$.
Let $c$ be a color in $L_{\vf}(q_0)\setminus\{\psi(u),\psi(v')\}$.  By Theorem~\ref{thm-thom}, there exists an $L$-coloring of $G_1-q_0$ such that
the color of $w$ is not $c$.  Observe that this coloring extends to an $L_{\vf}$-coloring of $G_1$ such that either $q_0$ or $q_1$
is colored by $c$.  This coloring extends to an $L_{\vf}$-coloring of $G_{\vf}$, which is a contradiction.

\item Therefore, we can assume that no two adjacent vertices of $G_{\vf}$ have lists of size three according to $L_{\vf}$.  By the minimality of $G$
and \refclaim{cl-v-dist}, we conclude that $G_{\vf}$ is one of the graphs drawn in Figure~\ref{fig-albimpob} (except \ob{V3'} which has vertex $x$ adjacent to $P$).
Let us discuss the possible cases for $G_{\vf}$ separately.

\begin{itemize}
\item If $G_\vf$ is \ob{V1}, \ob{V2a}, \ob{V3a} or \ob{V3'b}, then let $w$ be the vertex with $|L_{\vf}(w)|=3$ that is adjacent to both vertices of $P$.
By \refclaim{cl-v-chords}, we have that $w\in V(G)\setminus V(H)$, and thus $w$ is adjacent to two vertices in $\dom(\vf)$.
By \refclaim{cl-v-2chord}, these vertices are $v_1$ and $v_2$.  However, that contradicts \refclaim{cl-v-common}.

\item If $G_{\vf}$ is \ob{V2} or \ob{V3}, then $G$ contains a path $Q=p_0w_1w_2p_1$
corresponding to the outer face of $G_{\vf}$.  By \refclaim{cl-v-chords},
at most one of $w_1$ and $w_2$ belongs to $H$.  If say $w_1$ belongs to $H$,
then by \refclaim{cl-v-2chord}, $G$ consists of $G_{\vf}$ and a vertex with
a list of size three adjacent to $w_1$, $w_2$ and $p_1$.  However, such a graph
is $L$-colorable.  Therefore, neither $w_1$ nor $w_2$ belongs to $H$.
Let $F$ be the subgraph of $G$ split off by $Q$.
Since $s\ge 2$, $F$ has at least two vertices not belonging to $Q$, and by
Theorem~\ref{thm-5choosfar}, $F$ is \ob{P2}, \ob{P3}, \ob{P4} or \ob{P5}.
If $F$ is \ob{P2} or \ob{P4}, then we can assume that $Y=\{v_1\}$ is the
vertex with a list of size three; but then $v_1$ is adjacent to at most
one of $w_1$ and $w_2$, contrary to the fact that $|L_{\vf}(w_i)|<|L(w_i)|$ for $i\in\{1,2\}$.
Similarly, if $F$ is \ob{P5}, then we can assume that $Y=\{v_1,v_2\}$ is not adjacent to
at least one of $w_1$ and $w_2$, which is again a contradiction.  If $F$ is \ob{P3},
then $G_{\vf}$ is not \ob{V2}, since we assume that the common neighbor of the two vertices
in $Y$ is not adjacent to $x$. The final possibility, the combination of \ob{P3} and \ob{V3} does not result in a $P$-critical graph.

\item Suppose now that $G_{\vf}$ is \ob{V3'a}.  Let $p_iw_1w_2w_3p_{1-i}$ (for some $i\in\{0,1\}$)
be the subpath of $G$ corresponding to the outer face of $G_{\vf}$, where $|L_{\vf}(w_1)|=3$.
By \refclaim{cl-v-inside}, we have $w_2,w_3\not\in V(H)$.

Suppose that $w_1\not\in V(H)$.  Then $w_1$ is
adjacent to two vertices $y_1, y_2\in \dom(\vf)$, and by \refclaim{cl-v-2chord}, $p_iy_1y_2$ is a subpath of $H$
and $|L(y_1)|=3$.  Observe that $Y=\{y_1,y_2\}$, and thus $w_2$ and $w_3$ are adjacent to $y_2$.
By \refclaim{cl-v-2chord} applied to $y_2w_3p_{1-i}$, we conclude that $y_2$, $w_3$ and $p_{1-i}$
have a common neighbor with a list of size three (if $y_2$ were adjacent to $p_{1-i}$,
we would have chosen $Y=\{y_1\}$).
However, the resulting graph is $L$-colorable.  

Therefore, we have $w_1\in V(H)$.
By Theorem~\ref{thm-5choosfar}, the subgraph $F$ of $G$ split off by $Q=w_1w_2w_3p_{1-i}$
either consists of a vertex $z$ with a list of size three adjacent to all vertices of $Q$ ($z$ cannot
be adjacent to only three vertices of $Q$, since $Y=\{z\}$, $|L_\vf(w_2)|<|L(w_2)|$, and $|L_\vf(w_3)|<|L(w_3)|$),
or is equal to one of the graphs
drawn in Figure~\ref{fig-obst}. In the former case, $G$ is $L$-colorable.  In the latter case, the choice of $Y$
shows that $F$ is not \ob{P1}, and since $w_2$ and $w_3$ are adjacent to a vertex in $Y$, $F$ is not \ob{P4}, \ob{P5} or
\ob{P6}.  If $F$ is \ob{P3}, then the assumption that no neighbor of $x$ has two neighbors in $\dom(\vf)$ is violated,
and similarly we exclude the case that $F$ is \ob{P2} and $|L(v_1)|=3$.  The case that $F$ is \ob{P2} and $|L(v_1)|=4$
is excluded as well, since then $Y=\{v_2\}$ and $w_3$ is not adjacent to any vertex in $Y$.

\item Finally, suppose that $G_{\vf}$ is \ob{V4} and let $p_0w_1w_2w_3w_4p_1$ be the subpath of $G$ corresponding
to the outer face of $G_{\vf}$. By \refclaim{cl-v-inside}, we have $w_2,w_3\not\in V(H)$.

If $w_1\not\in V(H)$,
then $w_1$ is adjacent to two vertices $y_1,y_2\in \dom(\vf)$ and $y_2$ is also adjacent to $w_2$ and $w_3$.
In this case, $w_4$ cannot have two neighbors in $\dom(\vf)$, and since $|L_\vf(w_4)|=3$,
it follows that $w_4$ belongs to $V(H)$.  By \refclaim{cl-v-2chord}, either $y_2$ is adjacent to $w_4$, or
$w_4$, $w_3$ and $y_2$ have a neighbor with a list of size three.  However, in both cases $G$ would be $L$-colorable.

We conclude that $w_1\in V(H)$, and symmetrically $w_4\in V(H)$.  Let $F$ be the subgraph of $G$ split off by $w_1w_2w_3w_4$.
By Theorem~\ref{thm-5choosfar}, $F$ is either one of the graphs depicted in Figure~\ref{fig-obst}
or (taking into account that $|L_\vf(w_2)|<|L(w_2)|$ and $|L_\vf(w_3)|<|L(w_3)|$) consists of a vertex with a list of size three adjacent to $w_1$, $w_2$, $w_3$ and $w_4$.
In the latter case, $G$ is $L$-colorable; hence, consider the former.
Since both $w_2$ and $w_3$ are adjacent to a vertex in $Y$, $F$ is not \ob{P4}, \ob{P5} or \ob{P6}.
By the choice of $Y$, $F$ is not \ob{P1}.  And, if $F$ is \ob{P2} or \ob{P3}, then $x$ and two vertices
in $\dom(\vf)$ have a common neighbor.
\end{itemize}
\end{itemize}
\end{proof}

Let us consider a set $Y'\subseteq\{v_s,v_{s-1},\ldots\}$ and its partial coloring $\vf'$ chosen on the other side of $P$
by rules symmetric to the ones used to select $Y$ and $\vf$.  Clearly, the statement symmetric to \refclaim{cl-v-mains} holds:

\claim{cl-v-mains}{There exists a neighbor of $x$ adjacent to two vertices of $\dom(\vf')$.}

Let $w$ be the common neighbor of $x$ and two vertices $y_1,y_2\in Y$, where $|L(y_2)|=4$.
Let $w'$ be the common neighbor of $x$ and two vertices in $y'_1,y'_2\in Y'$, where $|L(y'_2)|=4$.

\claim{cl-v-choicey}{We can choose $y_1$ and $y'_1$ so that $|L(y_1)|=|L(y'_1)|=3$ and $y_1y_2, y'_1y'_2\in E(G)$.}
\begin{proof}
This only needs to be discussed in the case (Y4b), where $w$ could be a neighbor of $v_1$ and $v_3$,
but not $v_2$, and by \refclaim{cl-v-2chord}, $x$ would be contained inside the $4$-cycle $v_1v_2v_3w$ together with a common neighbor $z$ of
$v_1$, $v_2$ and $v_3$.  In that case, planarity implies that $w=w'$.
The choice of $Y$ implies that $Y'\neq Y$ (as we would then have $s=3$ and we would be in case (Y3)).
By \refclaim{cl-v-2chord} (applied to the 2-chord $y_1' w v_3$ and noting that $|L(y_2')|=4$), we have $y'_2=v_3$, $v_4$ is adjacent to $w$, and $y_1'$ is either $v_4$ or $v_5$.

Let $F$ be the subgraph of $G$
drawn inside the $4$-cycle $v_1v_2v_3w$.  Let $F'$ be the graph obtained from $F$ by splitting $w$ into two
vertices $w_1$ and $w_2$ adjacent to $x$ and by distributing the other neighbors of $w$ between $w_1$ and $w_2$
in the planar way.  Let $K=v_1v_2v_3w_1xw_2$ be the cycle bounding the outer face of $F'$, and note that $F'$ is $K$-critical
with respect to $L$.  By \refclaim{cl-v-inside}, $x$ is not adjacent to $v_1$ or $v_3$.  We conclude that
$K$ is an induced cycle.  No vertex of $F'$ other than $x$ is adjacent to both $w_1$ and $w_2$, hence
Lemma~\ref{lemma-critface} implies that $F-V(K)$ consists either of $z$ adjacent to $x$ and one of $w_1$ or $w_2$,
or of a triangle $zz_1z_2$, where $z_1$ is adjacent to $v_3$, $w_1$ and $x$, and $z_2$ is adjacent to $v_1$, $w_2$ and $x$.

In the latter case, note that $\deg_G(v_3)=5$ and observe that every $L$-coloring of $G-\{v_2,v_3,z, z_1,z_2\}$ extends
to an $L$-coloring of $G$, contrary to the assumption that $G$ is $P$-critical.  Therefore, assume that $z$ is the only vertex
of $V(F)\setminus V(K)$.  Choose a color $c\in S(z)\setminus L(v_2)$, let $G'=G-\{z,v_2\}$ and let
$L'$ be the list assignment obtained from $L$ by removing $c$ from the lists of $v_1$, $v_3$ and $w$.
Then $G'$ is $L'$-colorable by Lemma~\ref{lemma-noout} and this coloring extends to an $L$-coloring of $G$, which is a contradiction.
\end{proof}

\begin{figure}
\begin{center}
\includegraphics[width=120mm]{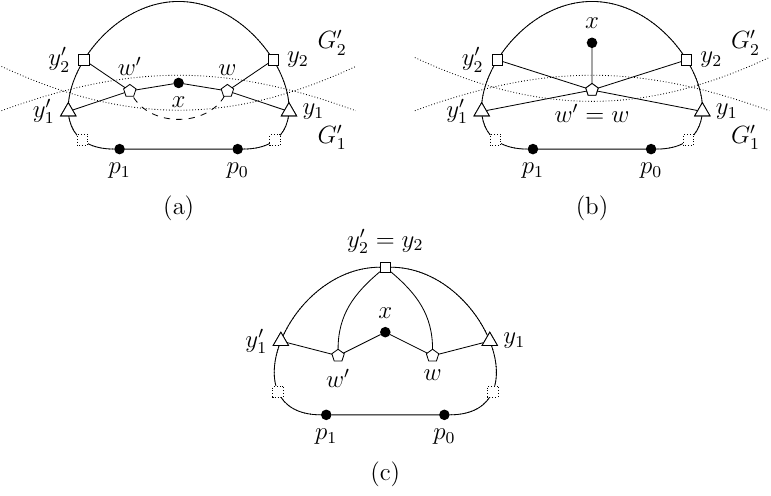}
\end{center}
\caption{Configrations following \refclaim{cl-v-choicey}.}
\label{fig-final}
\end{figure}

By the choice of $Y$ and $Y'$, note that $y_1\not\in \{y'_1,y'_2\}$ and $y'_1\not\in\{y_1,y_2\}$.  Furthermore, if $w=w'$, then \refclaim{cl-v-2chord}
implies that $x$ is contained in the subgraph of $G$ split off by $y_1wy'_1$.  If $w\neq w'$, then let $Q=Q_0=y_1wxw'y'_1$.
If $w=w'$, then let $Q$ be the star with center $w$ and rays $y_1$, $y'_1$ and $x$ and let $Q_0=y_1wy'_1$.  Let $G_2$ be the subgraph of $G$ split off by
$Q_0$ and let $G_1$ be the subgraph of $G$ such that $G_1\cup G_2=G$ and $G_1\cap G_2=Q_0$ (we have $P\subset G_1$).
Note that $y_2\not\in V(G_1)$.  If $w$ and $w'$ are adjacent in $G_2$, then let $G'_1=G_1+ww'$ and $G'_2=G_2-\{y_1,y'_1\}-ww'$,
otherwise let $G'_1=G_1$ and $G'_2=G_2-\{y_1,y'_1\}$.
See Figure~\ref{fig-final}(a) and (b).
Let $Q'=Q-\{y_1,y'_1\}$.  By the minimality of $G$, there exists an $L$-coloring $\theta$ of $G'_1$.

Suppose that $y_2\neq y'_2$.  We let $L_2$ be the list assignment such that $L_2(q)=\{\theta(q)\}$ for $q\in V(Q)$, $L_2(y_2)=L(y_2)\setminus\{\theta(y_1)\}$,
$L_2(y'_2)=L(y'_2)\setminus\{\theta(y'_1)\}$ and $L_2(v)=L(v)$ for all other vertices.  Note that all vertices of $G'_2-V(Q')$ have lists of size at least three.  Furthermore,
by \refclaim{cl-v-inside}, all neighbors of $x$ not in $Q'$ have lists of size five.  
Now we apply Lemma~\ref{lemma-extthom} to $G_2'$ and $L_2$. (If $w=w'$, we first split the edge $wx$ so that we obtain a precolored path of length 2.)
Lemma~\ref{lemma-extthom} implies that $G_2'$ contains an even fan procession for which $L_2$ is dangerous. In particular, since $x$ has no neighbors in $H$, we conclude that
$G'_2$ contains a fat fan of even order for which $L_2$ is a dangerous assignment.  Since $G$ does not have parallel edges, it follows that $w\neq w'$.  Since no two vertices with lists of size three are adjacent in $G$ and $ww'\notin E(G_2')$,
we conclude that the fat fan has order two and one of its vertices is $y_2$ (or $y_2'$).  Let $z\neq y_2$ be the vertex of the fat fan with $|L_2(z)|=3$.
By the choice of $Y$ and $Y'$, we have $z\neq y'_2$.  However, the $2$-chord $y'_1w'z$ then contradicts \refclaim{cl-v-2chord}.

We conclude that $y_2=y'_2$.  If $w=w'$, \refclaim{cl-v-septri} and \refclaim{cl-v-inside} imply that $w$ is the only neighbor of $x$.
However, Theorem~\ref{thm-onefour} then implies that $G$ is $L$-colorable.
It follows that $w\neq w'$.  By Lemma~\ref{lemma-critface} applied to $\cin_{w'xwy_2}(G)$ with list assignment $L$,
we have $V(\cin_{w'xwy_2}(G))=\{w',x,w,y_2\}$, and since $x$ is not adjacent to $y_2$ by \refclaim{cl-v-inside}
and $ww'\not\in E(G'_2)$ by the construction of $G'_2$, we conclude that
hence $G'_2$ is equal to the $4$-cycle $w'xwy_2$ and $\deg_G(y_2)=4$.
See Figure~\ref{fig-final}(c).

If $S(y_1)\not\subset L(y_2)$, then
we can color $y_1$ by a color in $S(y_1)\setminus L(y_2)$ and remove the color from the lists of neighbors of $y_1$,
obtaining a list assignment $L'$ for the graph $G'=G-\{y_1,y_2\}$.  Observe that $G'$ is not $L'$-colorable, and by the minimality of $G$,
one of the obstructions $Z$ depicted in Figure~\ref{fig-albimpob} appears in $G'$.  However, note that
$|L'(w')|=5$, since $y_1w'\notin E(G)$ by \refclaim{cl-v-septri} applied to $y_1y_2w'$.  It follows that either a vertex with a list of size five
or $x$ is incident with the outer face of $Z$. However, this does not happen for any of the obstructions in Figure~\ref{fig-albimpob}.

We conclude that $S(y_1)\subset L(y_2)$, and by symmetry $S(y'_1)\subset L(y_2)$.
Suppose that there exists a color $c\in S(y_1)\cap S(y'_1)$.
Let $L'$ be the list assignment for $G'=G-\{y_1,y'_1,y_2\}$ obtained by removing $c$ from the lists of neighbors of $y_1$ and $y'_1$.
Note that $G'$ is not $L'$-colorable.  By the minimality of $G$, one of the obstructions $Z$ depicted in Figure~\ref{fig-albimpob} appears in $G'$.
Since $x$ is not incident with the outer face of $Z$, both $w$ and $w'$ belong to $Z$ and $|L'(w)|=|L'(w')|=4$.
Together with \refclaim{cl-v-chords} (note that all vertices $z$ with $|L'(z)|=3$ belong to $V(H)$),
this implies that $Z$ is \ob{V2}, \ob{V3'a} or \ob{V4}. In all the cases, \refclaim{cl-v-2chord} uniquely determines $G$, and the resulting graph is $L$-colorable.
This is a contradiction.

It follows that $S(y_1)$ and $S(y'_1)$ are disjoint.  Since $|L(y_2)|=4$, we
conclude that $|S(y_1)|=|S(y_2)|=2$, and thus $y_1=v_1$ and $y'_1=v_s$, where
$s=3$.  Suppose that there exists a color $c\in S(w')\cap S(y_1)$.  Note that $c\not\in S(y'_1)$.
Let $G'=G-\{w',y_1,y_2\}$ with the list assignment $L'$ obtained from $L$ by removing $c$ from the lists
of neighbors of $w'$ and $y_1$, except for the vertex $y'_1$ where we set $L'(y'_1)=L(y'_1)$.  Note that
$y'_1$ is the only vertex with a list of size three and $x$ is incident with the outer face of $G'$, hence
by Lemma~\ref{lemma-noout}, $G'$ is $L'$-colorable.  However, this implies that $G$ is $L$-colorable, which is a contradiction.
We conclude that $S(w')\cap S(y_1)=\emptyset$, and symmetrically $S(w)\cap S(y'_1)=\emptyset$.

By symmetry, we can assume that $w$ has at most one neighbor in $P$, and thus $|S(w)|\ge 3$.  Since $S(y'_1)$ and $S(w)$ are
disjoint, $S(y'_1)\cup S(y_1)=L(y_2)$, $|S(y'_1)|=2$ and $|L(y_2)|=4$, there exists a color $c\in S(w)\setminus L(y_2)$. Clearly, $c\notin S(y_1)$.  Let $G'=G-\{w,y_2\}$ with the list assignment $L'$ obtained from $L$ by removing
$c$ from the lists of the neighbors of $w$ other than $y_1$ and $y'_1$.  Again, Lemma~\ref{lemma-noout} implies that $G'$ is $L'$-colorable,
giving an $L$-coloring of $G$.  This contradiction completes the proof of Lemma~\ref{lemma-albimps}.
\end{proof}

We are now ready to prove the main result.

\begin{proof}[Proof of Theorem~\ref{thm-main}]
By Lemma~\ref{lemma-albimps}, Lemma~\ref{lemma-albimp} holds with $M=2$.  
Let $X$ be the set of vertices with lists of size 1. The distance condition imposed in the theorem says that $X$ is $M$-scattered (for $M=2$). 
Since every planar graph is $5$-choosable, we can assume that $X\not= \emptyset$. We may also assume that a vertex $x\in X$ is incident with the outer face.  Furthermore, we can assume that $G$ is $\{x\}$-critical.
By Lemma~\ref{lemma-albcorr}, we conclude that all vertices of $G$ except for $x$ have lists of size at least $5$.  However, Theorem~\ref{thm-thom} with $P$ consisting of the vertex $x$ then implies that $G$ is $L$-colorable.
\end{proof}

\bibliographystyle{acm}
\bibliography{5choos}
\end{document}